\documentclass[12pt,reqno,a4paper]{amsart}
\usepackage{euscript,fullpage,amsmath,amssymb}
\usepackage[T1]{fontenc}
\usepackage{enumerate}

\newtheorem{thm}{Theorem}[section]
\newtheorem{lemma}[thm]{Lemma}
\newtheorem{remark}[thm]{Remark}

\newtheorem{rmk}[thm]{Remark}

\newtheorem{proposition}[thm]{Proposition}
\newtheorem{mainthm}{Theorem}

\newcommand{\N}{\mathbb{N}}
\newcommand{\R}{\mathbb{R}}
\newcommand{\Z}{\mathbb{Z}}

\newcommand{\mcl}{\mathcal L}

\newcommand{\bbp}{\mathbb P}
\newcommand{\C}{\mathbb{C}}

\newcommand{\al}{\alpha}

\newcommand{\ga}{\gamma}

\newcommand{\ep}{\epsilon}

\newcommand{\sig}{\sigma}
\newcommand{\Sig}{\Sigma}

\newcommand{\Lam}{\Lambda}
\newcommand{\Om}{\Omega}
\newcommand{\om}{\omega}

\def\Id{\text{\rm Id}}

\newcommand{\mc}[1]{\mathcal{#1}}

\DeclareMathOperator{\essinf}{ess\ inf}
\DeclareMathOperator{\esssup}{ess\ sup}

\newcommand{\paeom}{\ensuremath{\bbp\text{-a.e. } \om \in \Om}}
\newcommand{\lot}{\lambda_\omega^\theta} 

\title{A spectral approach for quenched limit theorems for random hyperbolic dynamical systems}
\date{\today}
\begin{document}
\begin{abstract}
We extend the recent spectral approach for quenched limit theorems developed for piecewise expanding dynamics under general random driving \cite{DFGTV} to quenched random piecewise hyperbolic dynamics including some classes of billiards.
For general ergodic sequences of maps in a neighbourhood of a hyperbolic map we prove a quenched large deviations principle (LDP), central limit theorem (CLT), and local central limit theorem (LCLT).
\end{abstract}
\maketitle
 \begin{center}
\authors{D.\ Dragi\v cevi\'c \footnote{Department of Mathematics, University of Rijeka, Rijeka Croatia. E-mail: {\tt \email{ddragicevic@math.uniri.hr}}.},
G.\ Froyland\footnote{School of Mathematics and Statistics,
University of New South Wales,
Sydney NSW 2052, Australia. Email: {\tt \email{G.Froyland@unsw.edu.au }}.},
C.\   Gonz\'alez-Tokman\footnote{School of Mathematics and Physics,
The University of Queensland,
St Lucia QLD 4072,
Australia. E-mail: {\tt \email{cecilia.gt@uq.edu.au}}.},
S.\ Vaienti\footnote{Aix Marseille Universit\'e, Universit\'e de Toulon, CNRS, CPT, 13009 Marseille, France. E-mail: {\tt \email{vaienti@cpt.univ-mrs.fr}}.}}

\end{center}
\tableofcontents
\section{Introduction}
\label{sec:intro}
In our previous paper \cite{DFGTV} we extended the Nagaev-Guivarc'h spectral method to obtain limit theorems, such as the Central Limit Theorem (CLT), the Large Deviation Principle (LDP) and the Local Central Limit Theorem (LCLT), for random dynamical systems governed by a \emph{cocycle} of maps $T^{(n)}_\omega:=T_{\sigma^{n-1}\omega}\circ\cdots\circ T_{\sigma\omega}\circ T_\omega$, assuming uniform-in-$\omega$ eventual expansivity conditions on the maps $T_\omega$.
The random driving was a general ergodic, invertible transformation $\sigma:\Omega\circlearrowleft$ on a probability space $(\Omega,\mathbb{P})$, and the real observable $g$ was defined on the product space $\Omega\times X\to\mathbb{R}.$

Before introducing our new results, we briefly recap the essence of the Nagaev-Guivarc'h spectral method in the deterministic setting,  where one deals with a single map  $T$, deferring to the original articles by Nagaev \cite{N57,N61} and Guivarc'h \cite{RE83,GH88} and to the  excellent survey \cite{gouezel2015limit} for more details.
The spectral method uses the transfer operator $\mathcal{L}:\mathcal{B}\circlearrowleft$ acting on a Banach space $\mathcal{B}$, and in particular, the \emph{twisted} transfer operator $\mathcal{L}^\theta f:=\mathcal{L}(e^{\theta g}f)$, for $f$ and $g\in\mathcal{B}$.
In the situation where $\mathcal{L}^\theta$ is quasi-compact for $\theta$ near zero, regularity of the leading eigenvalues and eigenprojectors have been used to prove limit theorems \cite{keller1992spectral,hennion,rey2008large,RE83,keller1992spectral,Broise, RE83,morita1994local,hennion,szasz2004local,gouezel2005berry} and more, namely Berry-Esseen theorems \cite{GH88,gouezel2005berry} and almost-sure invariance principles \cite{gouezel2010almost}. The key equality was $\mathbb{E}(e^{\theta S_ng} f)=\mathbb{E}((\mathcal{L}^{\theta})^n f),$ where $S_ng$ denotes the Birhkoff sum of the observable $g$ and the expectation is taken with respect to the unique eigenmeasure  $m$ of the adjoint of  $\mathcal{L}.$ Since the map $\theta \mapsto \mathcal{L}^{\theta}$ is holomorphic, classical perturbation theory allows one to obtain $\mathbb{E}(e^{\theta S_ng})=c(\theta) \lambda(\theta)^n+d_n(\theta),$ where $\lambda(\theta)$ is the leading eigenvalue of $\mathcal{L}^{\theta}$, with $c$ and $\lambda$ analytic in $\theta,$ and $\sup_{\theta}|d_n(\theta)|\rightarrow 0$.
We can therefore easily compute the characteristic function and the $\log$ generating function  of the process $g\circ T^n$ with respect to the invariant probability measure of $T$, which has density $dh/dm$ where $h$ is the eigenfunction of $\mathcal{L}$ corresponding to the simple eigenvalue 1.

In the quenched random setting we must replace the $n$-th power of the twisted operator with the twisted transfer operator cocycle $\mathcal{L}^{\theta, (n)}_\omega:=\mathcal{L}^{\theta}_{\sigma^{n-1}\omega}\circ\cdots\circ \mathcal{L}^{\theta}_{\sigma\omega}\circ \mathcal{L}^{\theta}_\omega.$ By using the multiplicative ergodic theorem adapted to the study of such cocycles and generalizing a theorem of Hennion and H\'erve \cite{hennion} to the random setting, we were able in our previous paper \cite{DFGTV} to show that the cocycle $\mathcal{L}^{\theta, (n)}_\omega$ is quasi-compact for $\theta$ near to $0.$ We therefore thus obtained that for such values of $\theta$ and for $\mathbb{P}$-a.e. $\omega\in \Omega$, the top Lyapunov exponent $\Lambda(\theta)$ (analogous to the logarithm of $\lambda(\theta)$ in the deterministic setting) of the cocycle is analytic and given by
$$
\lim_{n\rightarrow \infty} \frac1n \log|\mathbb{E}_{\mu_{\omega}}(e^{\theta S_ng(\omega, \cdot)})|=\Lambda(\theta),
$$
where $\mu_{\omega}$ is the  equivariant probability measure on the $\omega$-fiber (see below). This result together with the exponential decay of the norm of the elements in the complement of the top Oseledets space, which handled the error corresponding to quantity $d_n$ above, allowed us to achieve the desired limit theorems.

In the present paper we move from cocycles of piecewise expanding maps to cocycles of hyperbolic maps (both smooth and piecewise smooth), including some classes of billiards.  To our knowledge, this is the first time that this setting has been investigated with multiplicative ergodic theory tools.
One of the primary differences with \cite{DFGTV} is the use of anisotropic Banach spaces here in place of the space of functions of bounded variation in \cite{DFGTV}.
Specifically, in the smooth hyperbolic setting and in any dimension,  we use the functional analytic setup of Gou\"ezel and Liverani \cite{GL}, and in the piecewise hyperbolic case in dimension two we use the spaces from Demers and Liverani \cite{DL} (as well as Demers and Zhang~\cite{DM1, DM2}).
This increased technicality in the underlying spaces necessitates a certain amount of checking of relevant conditions, however, we wish to highlight the fact that a wholesale change of the theory of \cite{DFGTV} is not required, which demonstrates the power and flexibility of our approach.
The use of transfer operators in the study of statistical properties and limit theorems for hyperbolic dynamical systems has flourished in the last years, and \cite{BaladiUltimate} presents a thorough discussion of the various spaces that have been used in the literature.  Our intention in this work has not been to find the most general version of the results, but rather to illustrate the applicability of the methods. In fact, we expect the methods presented here to remain applicable in some (or all) of these functional analytic scenarios.

We first consider cocycles $T^{(n)}_\omega$ where the family of maps $\{T_\omega\}_{\omega\in\Omega}$ are selected from a $C^{r+1}$-neighbourhood of a topologically transitive Anosov map $T$ of class $C^{r+1}$ (in Section \ref{sect:pwh} we consider piecewise hyperbolic maps also describing periodic Lorentz gas).
The random driving $\sigma:\Omega\circlearrowleft$ is a general (ergodic, invertible) automorphism preserving a probability measure $\mathbb{P}$.
If $d_{C^{r+1}}(T_\omega,T)<\Delta$ for $\mathbb{P}$-a.e.\ $\omega\in\Omega$ and $\Delta$ is sufficiently small,
the random dynamical system generated by the cocycle $T^{(n)}_\omega$ supports a measure $\mu$, invariant under the skew product $\tau(\omega,x)=(\sigma\omega,T_\omega x)$.
We obtain this measure by explicitly constructing the family $\mu_{\omega}$ along the marginal $\mathbb{P}$, namely $\mu=\int_\Omega \mu_\omega\ d\mathbb{P}(\omega),$ and satisfying the usual equivariance condition $\mu_{\omega}\circ T_\omega^{-1}=\mu_{\sigma\omega}$.
Our observable $g$ satisfies $g(\omega,\cdot)\in C^r$ for $\mathbb{P}$-a.e.\ $\omega$,  $\esssup_{\omega\in\Omega} \|g(\omega,\cdot)\|_{C^r}<\infty$, and is fiberwise centred:  $\int_X g(\omega,x)\ d\mu_\omega(x)=0$ for $\mathbb{P}$-a.e.\ $\omega$.
Our limit theorems concern random Birkhoff sums
\begin{equation}
\label{birkhoffsums}
S_ng(\omega,x):=\sum_{i=0}^{n-1}g(\tau^i(\omega,x))=\sum_{i=0}^{n-1}g(\sigma^i,T_\omega^{(i)}x),\quad (\omega,x)\in \Omega\times X, n\in \mathbb{N}.
\end{equation}
Our main theorems  are:
\begin{mainthm}[Quenched large deviations theorem]\label{thm:ldt}
In the above setting, there exists $\epsilon_0>0$ and a non-random function $c\colon (-\epsilon_0, \epsilon_0) \to \mathbb R$ which is nonnegative, continuous, strictly convex, vanishing only at $0$ and such that
\[
\lim_{n\to \infty} \frac 1 n \log \mu_\om(S_n g(\om, \cdot ) >n\epsilon)=-c(\epsilon), \quad \text{for  $0<\epsilon <\epsilon_0$ and \paeom}.
\]
\end{mainthm}
We now define the non-random quantity
\begin{equation}\label{variance}
     \Sig^2 :=\int _{\Om \times X} g(\om, x)^2\, d\mu(\om, x)+2 \sum_{n=1}^\infty \int_{\Om \times X}  g(\om, x) g(\tau^n (\om, x))\, d\mu(\om, x).
    \end{equation}
It is clear that $\Sig^2\ge 0$.
\begin{mainthm}[Quenched central limit theorem] \label{thm:clt}
In the above setting, assume that the non-random \textit{variance} $\Sig^2$, defined in \eqref{variance} satisfies $\Sig^2>0$.
Then, for every bounded and continuous function $\phi \colon \R \to \R$ and \paeom, we have
\[
\lim_{n\to\infty}\int \phi \bigg{(}\frac{S_n g(\om, x)}{ \sqrt n}\bigg{)}\, d\mu_\om (x)=\int \phi \, d\mathcal N(0, \Sig^2).
\]
(The discussion in \S\ref{sec:convexity} deals with the degenerate case $\Sig^2=0$).
\end{mainthm}
One of the main achievements of our previous paper was the proof of the local central limit theorem (LCLT) in the non-arithmetic and arithmetic cases. Our basic assumption, which for convenience we simply call (L), expresses the exponential decay of the strong norm of the twisted operator when the parameter $\theta=it$ has  $t\neq 0.$ Moreover we showed under additional assumptions that we will recall in section 8, that hypothesis  (L) was equivalent to a   co-boundary condition which is better known as the {\em aperiodicity condition}. In the present paper we prove the LCLT in the non-arithmetic case by assuming (L).
Recently Hafouta and Kifer \cite{HafoutaKifer} proposed a new set of assumptions which allow us to check condition (L). We will see that some of these assumptions can be verified  easily for our systems, provided we restrict the class of the driving maps.

\begin{mainthm}[Quenched local central limit theorem] \label{thm:lclt}
In the above setting  suppose condition (L) holds.
Then, for \paeom\  and every bounded interval $J\subset \R$, we have
 \[
  \lim_{n\to \infty}\sup_{s\in \R} \bigg{\lvert} \Sig \sqrt{n} \mu_\om (s+S_n g(\om, \cdot)\in J)-\frac{1}{\sqrt{2\pi}}e^{-\frac{s^2}{2n\Sig^2}}\lvert J\rvert \bigg{\rvert}=0.
 \]
\end{mainthm}

In Section \ref{sect:pwh} we consider random cocycles of \emph{piecewise} hyperbolic maps of the type considered in \cite{DL} on two-dimensional compact Riemannian manifolds, and billiard maps associated to periodic Lorentz gas \cite{DM1, DM2}. As we will explain later on and in order to apply the multiplicative ergodic theorem, we have now less choice for the random distribution of the maps, but for instance we can deal  with countably many maps. All the preceding theorems~\ref{thm:ldt}, \ref{thm:clt} and \ref{thm:lclt} still hold.\\

Apart from \cite{DFGTV} there are some quenched limit theorems (LDP and CLT) that have been obtained using different methods.
Kifer derives a large deviation principle \cite{kifer1990large,kifer1992equilibrium,kifer1998limit}  for occupational measures using theory of equilibrium states, and a central limit theorem  via martingale methods;  in both cases, he treats random subshifts of finite type and random smooth expanding maps.
Recently, Hafouta and Kifer \cite{HafoutaKifer} proved limit theorems  for these systems in the more general ``nonconventional setting''.
They used (complex) cone techniques, where the cones were defined in the functional space upon which the transfer operator acts. We emphasize that  they don't consider the case of hyperbolic dynamics studied in the present paper. In fact, is not clear if their cone techniques can be adapted to the present setting.
Bakhtin \cite{bakhtin1994random} is probably the closest to our work; he proves a central limit theorem and large deviation estimates for mixing sequences of smooth hyperbolic maps with common expanding and contracting directions, under a variance growth condition on the Birkhoff sums. He also used cones, but living on the tangent space of the manifold.
In comparison to Bakhtin,  we can additionally treat the case of random piecewise hyperbolic maps (with singularities and including billiards), and moreover we exhibit explicitly the rate function which produces asymptotic large deviation bounds; the local CLT is also new in this setting.

\section{Preliminaries}
\label{sec:prelim}

Let $X$ be a $d$-dimensional $C^\infty$ compact connected Riemannian manifold and let $T$ be a topologically transitive Anosov map of class $C^{r+1}$, where $r>2$.
We follow the setup of \cite{GL}.
Replacing the Riemannian metric by an adapted metric \cite{mather}, we use hyperbolicity constants $0<\nu<1<\lambda$, where $\lambda$ is less than the minimal expansion along the unstable directions, $\nu$ is greater than the minimal contraction along the stable directions,
 and the angles between the stable and unstable spaces (of dimensions $d_s, d_u$, respectively) are close to $\pi/2$.
A collection of $C^\infty$ coordinate charts $\psi_i:(-r_i,r_i)^d\to X$, $i=1,\ldots,N$ are defined so that $\bigcup_{i=1}^N \psi_i((-r_i/2,r_i/2)^d)$ cover $X$, with the $r_i$ small enough that $D\psi_i(0)\cdot (\mathbb{R}^{d_s}\times \{0\})=E^s(\psi_i(0))$, $|\psi_i|_{C^{r+1}}, |\psi_i^{-1}|_{C^{r+1}}\le 1+\kappa$, and $\kappa$ small enough in such a way  that the stable cone at $x$ in $\mathbb{R}^d$ is compatibly mapped to the stable cone at $\psi_i(x)$ in $X$ (see \cite{GL} for details).
Let $G_i(K)$ denote the set of graphs of $C^{r+1}$ functions $\chi:(-r_i,r_i)^{d_s}\to (-r_i,r_i)^{d_u}$ with $|\chi|_{C^{r+1}}\le K$ (and with $|D\chi|\le c_i$ so that the tangent space of the graph belongs to the stable cone in $\mathbb{R}^d$ mentioned above).
For large enough $K$, the coordinate map $\psi_j^{-1}\circ T^{-1}\circ \psi_i$  maps $G_i(K)$ into $G_j(K')$ for some $K'<K$.
For $A$ sufficiently large, (depending on $\kappa$ and $\nu$) and $\delta$ small enough that $A\delta<\min_i r_i/6$, an admissible graph is a map $\chi:\bar{B}(x,A\delta)\to (-2r_i/3,2r_i/3)^{d_u},$ $\bar{B}(x,A\delta)\subset (-2r_i/3,2r_i/3)^{d_s}$;  the collection of admissible graphs is denoted $\Xi_i$.

 For $p\in \mathbb N$, $p\le r$, $q\ge 0$ and $h\in C^r(X,\mathbb{C})$, $\varphi\in C^q(X,\mathbb{C})$ we define (using the notation in~\cite{GL})
\begin{equation}\label{auxnorms}
 \lVert h \rVert^{\sim}_{p, q}:=\sup_{\substack{\lvert \alpha \rvert=p\\ 1\le i\le N} } \sup_{\substack{\chi \colon \overline{B}(x, A\delta)\to \mathbb R^{d_u} \\ \chi \in \Xi_i}}
 \sup_{\substack{\varphi \in C_0^q(\overline B(x, \delta), \mathbb C) \\ \lvert \varphi \rvert_{C^q} \le 1}}
 \bigg{\lvert}\int_{B(x, \delta)}\big{[}\partial^{\alpha} (h\circ \psi_i)\big{]} \circ (\mathcal{\Id}, \chi)\cdot \varphi \bigg{\rvert}.
\end{equation}
Finally, for $p$ and $q$ as above satisfying $p+q<r$, we set
\begin{equation}\label{norms}
 \lVert h\rVert_{p, q}:=\sup_{0\le k\le p}\lVert h \rVert^{\sim}_{k, q+k}=\sup_{p' \le p, q'\ge q+p'}\lVert h \rVert^{\sim}_{p', q'}.
\end{equation}
The space $\mathcal B^{p, q}$ is defined to be the completion of $C^r(X, \mathbb C)$ with respect to the norm $\lVert \cdot \rVert_{p, q}$.\\ The following proposition will be useful when applying the multiplicative ergodic theorem.

\begin{proposition}
 The space $\mathcal B^{p, q}$ is separable.
\end{proposition}

\begin{proof}
 The desired conclusion follows directly from~\cite[Remark 4.3]{GL} after we note that $C^\infty (X, \mathbb C)$ has a countable subset which is dense with respect to the
 $C^r$ norm.
\end{proof}
We recall from~\cite[Section 4]{GL} that the elements of $\mathcal B^{p, q}$ are distributions of order at most $q$. More precisely, there exists $C>0$ such that   any $h\in \mathcal B^{p, q}$ induces a
linear functional
$\varphi \to h(\varphi)$ with the property that
\begin{equation}
\label{dist_cont}
 \lvert h(\varphi)\rvert \le C\lVert h\rVert_{p, q} \lvert \varphi\rvert_{C^q}, \quad \text{for $\varphi \in C^q(X,\mathbb{C})$.}
\end{equation}
In particular, for $h\in C^r$ we have that
\begin{equation}
\label{dcont}
h(\varphi)=\int_X h\varphi,  \quad \text{for $\varphi \in C^q(X,\mathbb{C})$.}
\end{equation}
We say that $h\in \mathcal B^{p, q}$ is nonnegative and write $h\ge 0$ if $h(\varphi)\ge 0$ for any $\varphi \in C^q(X,\mathbb{R})$ such that $\varphi \ge 0$.

Let $\mcl_T \colon \mathcal B^{p, q} \to \mathcal B^{p, q}$ be the transfer operator associated to $T$ defined by
\begin{equation}
\label{aux_cont}
  (\mcl_T h)(\varphi)=h(\varphi \circ T), \quad \text{for $h\in \mathcal B^{p, q}$ and $\varphi \in C^q(X,\mathbb{C})$.}
\end{equation}
We recall that for  $h\in C^r(X, \mathbb C)$, $\mathcal{L}_T$ is the function given by
\begin{equation}\label{tocont}
 \mcl_T h=\bigg{(}\frac{h} {\lvert \det T\rvert}\bigg{)}\circ T^{-1}.
\end{equation}
Take $g\in C^r(X, \mathbb C)$ and $h\in \mathcal{B}^{p,q}$. Then, there exists a sequence $(h_n)_n \subset C^r(X, \mathbb C)$ that converges to $h$ in $\mathcal{B}^{p,q}$.  It follows  that $(gh_n)_n \subset C^r(X, \mathbb C)$ is a Cauchy sequence in $\mathcal{B}^{p,q}$ and therefore it converges to some element of $\mathcal{B}^{p,q}$ which
we denote by $g\cdot h$. It is straightfoward to verify that the above construction does not depend on the particular choice of the sequence $(h_n)_n$. Moreover, the action of $g\cdot h$ as a distribution is given by
\begin{equation}\label{1145}
(g\cdot h)(\varphi)=h(g\varphi), \quad \varphi \in C^q(X, \mathbb C).
\end{equation}
We will need  the following result.
\begin{lemma}
\label{duallemma}
For $h\in\mathcal{B}^{p,q}, g\in C^r(X,\mathbb{C})$ one has $\mathcal{L}_T(g\circ T\cdot h)=g\cdot\mathcal{L}_Th$.
\end{lemma}
\begin{proof}
Let $\varphi\in C^q(X,\mathbb{C})$. It follows from~\eqref{aux_cont} and~\eqref{1145} that
 $[\mathcal{L}_T(g\circ T\cdot h)](\varphi)=(g\circ T\cdot h)(\varphi\circ T)=h(g\circ T\cdot\varphi\circ T)=\mathcal{L}_Th(g\cdot\varphi)=(g\cdot\mathcal{L}_Th)(\varphi)$,
which yields the desired result.
\end{proof}

\section{Building the cocycle $\mcl$} \label{sec:InitialCocycle}
In the sequel we will consider the case $p=q=1$ and $r>2$, but we will also require $T$ to be $C^{r+1}$, to be in a suitable framework for perturbations. Using the fact that the unit ball in $\mathcal B^{1, 1}$ is relatively compact in $\mathcal B^{0, 2}$ \cite[Lemma 2.1]{GL}, it follows from~\cite[Theorem 2.3]{GL} that the associated transfer operator
$\mcl_T$ is quasicompact on $\mathcal B^{1,1}$, $1$ is a simple
eigenvalue and there are no other eigenvalues of modulus $1$. This in particular implies (using the terminology as in~\cite[Definition 2.6]{CR}) that $\mcl_T$ is exact in
$\{h\in \mathcal B^{1, 1}: h(1)=0\}$.
Let
\[
\mathcal M_\ep(T)=\{S : \ \text{$S$ is an Anosov map of class $C^{r+1}$ satisfying $d_{ C^{r+1}}(S, T)<\epsilon$}\}.
\]
We also recall (see~\cite[Lemmas 2.1. and 2.2]{GL} and the discussion at the beginning of \S7 \cite{GL}) that there exist $\epsilon, A>0$ and $c\in (0, 1)$ such that for any $T'\in \mathcal M_\ep(T)$, one has
\begin{itemize}
 \item  $\lVert \mcl_{T'}^n h\rVert_{0,2}\le A\lVert h\rVert_{0, 2}$ for each $n\in \N$ and $h\in \mathcal B^{1,1}$;
 \item
 $
  \lVert \mcl_{T'}^n h\rVert_{1,1}\le Ac^n \lVert h\rVert_{1,1}+A\lVert h\rVert_{0, 2}$ for each $n\in \N$ and $h\in \mathcal B^{1,1}$.
\end{itemize}
Consider the family of (bounded, linear) transfer operators, acting on the Banach space $(\mathcal B^{1, 1}, \|\cdot\|)$,
\[
\mathcal O_\ep(T, \mathcal B^{1, 1})=\{\mcl_{S}: \mathcal B^{1, 1} \to \mathcal B^{1, 1} \text{ such that }  S \in \mc M_\ep(T)\}.
\]
It follows from~\cite[Proposition 2.10]{CR} (applied to the case where $\lVert \cdot \rVert=\lVert \cdot \rVert_{0,2}$ and $\lvert \cdot \rvert_v=\lVert \cdot \rVert_{1,1}$)
that there exists $0<\epsilon_0 \le \epsilon$, $D, \lambda >0$ such that for any $\mathcal{L}_{T_1}, \ldots, \mathcal{L}_{T_n}\in \mathcal O_{\ep_0}(T,\mathcal{B}^{1,1})$, one has
\begin{equation}\label{dec_cont}
 \lVert \mcl_{T_n}\circ\cdots\circ \mcl_{T_2}\circ\mcl_{T_1}h\rVert_{1,1} \le De^{-\lambda n}\lVert h\rVert_{1,1} \quad \text{for $h\in \mathcal B^{1,1}$ satisfying $h(1)=0$,}
\end{equation}
where $\mcl_{T_i}$ denotes the transfer operator associated with $T_i$.
From now on, we replace $\epsilon_0$ by $\epsilon$ so that (\ref{dec_cont}) holds on $\mathcal O_\ep(T, \mathcal B^{1, 1})$.

We now build the cocycle $\mc{R}:=(\Om, \mc F, \bbp, \sig, \mathcal B^{1, 1}, \mcl)$, simply referred to as $\mcl$,
as follows:
\begin{enumerate}
\item
Let $(\Om, \mc F, \bbp)$ be a  probability space, where $\Om$ is a Borel subset of a separable, complete metric space and $\sig: \Om \to \Om$ an ergodic, invertible $\bbp$-preserving transformation.
\item
Let $\mc{T}: \Om \to \mathcal M_\ep(T)$ be a  measurable map given by $\om \mapsto T_\om$.
By applying~\cite[Lemma 7.1]{GL} we find that there exists $C>0$ such that for any $S \in \mc M_\ep(T)$,
\[
 \sup_{\lVert h\rVert_{1,1}\le 1}\lVert (\mcl_{S}-\mcl_T)h\rVert_{0, 2} \le C\epsilon.
\]
\end{enumerate}

\subsection{Strong measurability of $\omega\mapsto\mcl_\omega$}\label{S:strongMeas}

In this section we demonstrate strong measurability of the map $\mcl: \Om \to \mc{O}_\ep (T,\mathcal B^{1, 1})$ given by $\om \mapsto \mcl_\om:= \mcl_{T_\om}$;  this is required to establish the existence of measurable Oseledets spaces for the cocycle.
 To prove strong measurability of  $\om \mapsto \mcl_\om:= \mcl_{T_\om}$, we will show that the map from $\mc{M}_\ep(T)$ to $\mc{B}^{1,1}$ defined by $S \mapsto \mcl_S$ is strongly continuous.
For this, let $S \in \mc{M}_\ep(T)$ and $h \in \mathcal B^{1, 1}$. We must show that
$\| \mcl_{\tilde{S}}h - \mcl_S h\|_{1,1} \to 0$ as $d_{C^{r+1}}(\tilde{S},S)\to0$.
First, assume $h\in  C^r$.
Then, we need to estimate differences of the form
$$\bigg{\lvert}\int_{B(x, \delta)}\big{[}\partial^{\alpha} (\mcl_S h\circ \psi_i)\big{]} \circ (\mathcal{\Id}, \chi)\cdot \varphi -\int_{B(x, \delta)}\big{[}\partial^{\alpha} (\mcl_{\tilde{S}} h\circ \psi_i)\big{]} \circ (\mathcal{\Id}, \chi)\cdot \varphi \bigg{\rvert},$$
where
$\alpha, \chi$ and $\varphi$ vary as in the definition in \eqref{auxnorms}, with $p=q=1$.
Arguing as in \cite[Lemma 7.1]{GL}, and employing the corresponding notation, we write
\begin{equation}\label{eq:transfOpCharts}
\int_{B(x, \delta)}\big{[}\partial^{\alpha} (\mcl_S h\circ \psi_i)\big{]} \circ (\mathcal{\Id}, \chi)\cdot \varphi=
\sum_{|\beta|\leq |\alpha|} \sum_{j=1}^l
 \int_{B(x_j, \delta)}\partial^{\beta} \tilde{h}_j \circ (\mathcal{\Id}, \chi_j) \cdot F_{\alpha, \beta, S, j} \cdot \rho_j,
\end{equation}
where
$\chi_1, \dots, \chi_l$ are $\gamma$-admissible graphs whose corresponding $\gamma$-admissible leaves cover $S^{-1}(W)$,
with $W$ an admissible leaf corresponding to the graph of $\chi$;
$\tilde{h}_j = h \circ \psi_{i(j)}$; $\{\rho_j\}_{j=1, \dots, l}$ is a partition of unity subordinated to the $\gamma$-admissible leaves of $\chi_j$; and $F_{\alpha, \beta, S, j}$ are functions bounded in $C^{q+|\beta|}$.
A similar expression holds for $\int_{B(x, \delta)}\big{[}\partial^{\alpha} (\mcl_{\tilde{S}} h\circ \psi_i)\big{]} \circ (\mathcal{\Id}, \chi)\cdot \varphi$, with $F_{\alpha, \beta, S, j}$ replaced by $F_{\alpha, \beta, \tilde{S}, j}$ and $\chi_j$ replaced by $\tilde{\chi}_j$, the graph corresponding to $\psi_{i(j)}^{-1} \circ \tilde{S}^{-1}\circ S \circ \psi_{i(j)} \circ (Id, \chi_j) (B(x_j,  \gamma A \delta) )$.
Furthermore, if $d_{C^{r+1}}(S, \tilde{S})$ is small enough, each $\tilde{\chi_{j}}$ is a graph in $\Xi_{i(j)}$, and $|\chi_{j}-\tilde{\chi_{j}}|_{C^2(\bar B(x_j,  A \delta))}< C d_{C^{r+1}}(S, \tilde{S})$.
Also,
$\|F_{\alpha, \beta, S, j}\|_{C^{q+|\beta|}}, \| F_{\alpha, \beta, \tilde{S}, j}\|_{C^{q+|\beta|}}$ are uniformly bounded for
 $S, \tilde{S} \in \mc{M}_\ep(T)$ and
$\|F_{\alpha, \beta, S, j}- F_{\alpha, \beta, \tilde{S}, j}\|_{C^{q+|\beta|}} \to 0$ as $d_{C^{r+1}}(\tilde{S},S)\to0$, uniformly over $\varphi$ as in \eqref{auxnorms}.
Hence, as $d_{C^{r+1}}(\tilde{S},S)\to0$, we get
\[
\Big| \int_{B(x_j, \delta)}
\partial^{\beta} (\tilde{h}_j) \circ (\mathcal{\Id}, \tilde{\chi}_j) \cdot F_{\alpha, \beta, \tilde{S}, j} \cdot \rho_j -
\partial^{\beta} (\tilde{h}_j) \circ (\mathcal{\Id}, \chi_j) \cdot F_{\alpha, \beta, S, j} \cdot \rho_j
\Big | \to 0,
\]
uniformly over $\chi$ (and so $\chi_j$) and $\varphi$ as in \eqref{auxnorms}.
It then follows from \eqref{eq:transfOpCharts} that $\| \mcl_{\tilde{S}}h - \mcl_S h\|_{1,1} \to 0$ as $d_{C^{r+1}}(\tilde{S},S)\to0$, as claimed.

The result for general $h \in \mathcal{B}^{1,1}$ follows from an approximation argument by $C^r$ functions, because if $d_{C^{r+1}}(S,\tilde{S})$ is sufficiently small,  then $\|\mcl_{\tilde{S}}\|_{1,1}\leq 1+ \|\mcl_{S}\|_{1,1}=:M$. Indeed,
let $\{h_j\}_{j\in \N}$ be a sequence of  $C^r$ functions such that $\lim_{j\to \infty}h_j =h$ in $B^{1,1}$, and let $\ep>0$.  Then, there exists $n\in \N$ such that $\|h-h_n\|_{1,1}<\frac{\epsilon}{3M}$.
Hence, $\| \mcl_{\tilde{S}}h - \mcl_S h\|_{1,1} \leq \| \mcl_{\tilde{S}}h - \mcl_{\tilde{S}} h_n\|_{1,1}+\| \mcl_{\tilde{S}}h_n - \mcl_S h_n\|_{1,1}+\| \mcl_{S}h_n - \mcl_S h\|_{1,1} \leq \frac{2\ep}{3} +\| \mcl_{\tilde{S}}h_n - \mcl_S h_n\|_{1,1}$.
Since $h_n\in C^r$, we have that    $\lim\sup_{d_{C^{r+1}}(\tilde{S},S)\to 0}\| \mcl_{\tilde{S}}h - \mcl_S h\|_{1,1} \leq  \frac{2\ep}{3}$.
Since the choice of $\ep>0$ is arbitrary, the result follows.

 \subsection{Quasi-compactness of the cocycle $\mathcal{L}$ and existence of Oseledets splitting}
 \label{sect:qc0}

 We may apply Kingman's subadditive ergodic theorem to form the following limits, which are constant for $\mathbb{P}$-a.e.\ $\omega\in\Omega$:
 $$\Lambda(\mathcal R):=\lim_{n\to\infty}\frac{1}{n}\log\|\mathcal L^{(n)}_\omega\|_{1,1},\mbox{ and}$$
 $$\kappa(\mathcal R):=\lim_{n\to\infty}\frac{1}{n}\log ic(\mathcal L^{(n)}_\omega), \quad \text{where}$$
 $$ic(A):=\inf\{r>0: A(B_{\mathcal{B}_{1,1}})\mbox{ can be covered with finitely many balls of radius } r\},$$
 and $B_{\mathcal{B}_{1,1}}$ is the unit ball in $\mathcal{B}_{1,1}$.
 The cocycle $\mathcal{R}$ is called \emph{quasi-compact} if $\Lambda(\mathcal R)>\kappa(\mathcal R)$.

For each $\om \in \Om, n\in \N$, let $\mcl_\om^{(n)} := \mcl_{\sig^{n-1}\om} \circ \dots \circ \mcl_{\sig\om}\circ \mcl_\om$.
It follows readily from~\eqref{dec_cont} that
\begin{equation}\label{DEC_cont}
 \lVert \mcl_\om^{(n)} h\rVert_{1,1} \le De^{-\lambda n} \lVert h\rVert_{1,1} \quad \text{for any $\om \in \Om$, $n\in \mathbb N$ and $h\in \mathcal B^{1,1}$, $h(1)=0$.}
\end{equation}
After possibly decreasing $\ep$ we can also assume\footnote{See the discussion at the beginning of \S7 in \cite{GL}.} that there exist $a\in (0, 1)$ and $B, K>0$ such that for every $\om \in \Om$, $n\in \mathbb N$ and $h\in \mathcal B^{1,1}$,
\begin{equation}\label{wsly_cont}
 \lVert \mcl_\om^{(n)} h\rVert_{0,2} \le B\lVert h\rVert_{0, 2}, \quad   \lVert \mcl_\om^{(n)} h\rVert_{1,1} \le Ba^n\lVert h\rVert_{1,1}+B\lVert h\rVert_{0,2},
\end{equation}
which in particular implies that
\begin{equation}\label{ub_cont}
 \lVert \mcl_\om h\rVert_{1,1} \le K\lVert h\rVert_{1,1},
\end{equation}
where $K:=Ba+B>0$.
By  \cite[Lemma 2.1]{DFGTV}, the inequalities (\ref{wsly_cont}) and (\ref{ub_cont}) imply that the cocycle $\mathcal{R}$ is quasi-compact.
By separability of $\mathcal{B}^{1,1}$, and quasi-compactness and strong measurability of $\mathcal{L}$, the multiplicative ergodic theorem (Theorem A, \cite{GTQuas1}) yields the existence of a measurable Oseledets splitting
$$\mathcal{B}^{1,1}=\left(\bigoplus_{j=1}^l Y_j(\omega)\right)\oplus V(\omega),$$
where each component of the splitting is equivariant under $\mathcal{L}_\omega$.
The $Y_j(\omega)$ are finite-dimensional and by $0=\lambda_1>\lambda_2>\ldots$ we denote the corresponding (finite or infinite) sequence of Lyapunov exponents.

\subsection{One-dimensionality of the top Oseledets space}
\label{sec:topspace}
\begin{proposition}\label{prop:uniqueAcim}
There exists a unique family  $(h_\om^0)_{\om \in \Om} \subset \mathcal B^{1,1}$ such that:
 \begin{enumerate}
 \item $\mcl_\om h_\om^0=h_{\sigma \om}^0$ for  \paeom;
 \item $h_\om^0$ is nonnegative and $h_\om^0(1)=1$ for \paeom;
 \item $\om \to h_\om^0$ is a measurable map from $\Om$ to $\mathcal B^{1,1}$;
 \item \begin{equation}\label{b}\esssup_{\om \in \Om} \lVert h_\om^0\rVert_{1,1} <\infty.
     \end{equation}
 \end{enumerate}
\end{proposition}

\begin{proof}
 Let
 \[
  Y=\{ v\colon \Om \to \mathcal B^{1,1}: \ \text{$v$ is measurable and $\lVert v\rVert_\infty:=\esssup_{\om \in \Om} \lVert v(\om)\rVert_{1,1} <\infty$} \}.
 \]
Then, $Y=(Y, \lVert \cdot\rVert_\infty)$ is a Banach space. Furthermore, let $Z$ be the subset of $Y$ that consists of $v\in Y$ with the property that $v(\om)$ is nonnegative
and $v(\om)(1)=1$ for \paeom.  It is easy to verify that $Z$ is a closed subset of $Y$. Indeed, assume that $(v_n)_{n\in \Z}$ is a sequence in $Z$ converging to $v\in Y$. It follows from~\eqref{dist_cont} that
\[
 \lvert v(\om)(\varphi)-v_n(\om)(\varphi)\rvert \le C\lVert v_n(\om)-v(\om)\rVert_{1,1} \lvert \varphi \rvert_{C^1}\le C\lVert v_n-v\rVert_\infty \lvert \varphi \rvert_{C^1},
\]
and thus $v_n(\om)(\varphi) \to v(\om)(\varphi)$ for $\varphi \in C^1$ and \paeom. Thus, $v(\om)(\varphi)\ge 0$ for $\varphi \ge 0$ and $v(\om)(1)=1$ for \paeom \ and we conclude that $v\in Y$.

We define $\mathbb L \colon Z \to Z$ by
\[
 (\mathbb L v)(\om)=\mcl_{\sigma^{-1} \om}v(\sigma^{-1} \om) \quad \om \in \Om, \ v\in Z.
\]
It follows from \eqref{aux_cont} and~\eqref{ub_cont} that $\mathbb L$ is a well-defined and continuous map on $Z$. Using~\eqref{DEC_cont}, one can easily verify (see~\cite[Proposition 1]{DFGTV2}) that there exists $n_0\in
\mathbb N$ such that $\mathbb L^{n_0}$ is a contraction on $Z$. Thus, $\mathbb L$ has a  unique fixed point $\bar{v}\in Z$. It is easy to verify that the family $h_\omega^0$, $\omega \in \Omega$ defined $h_\omega^0=\bar{v}(\omega)$, $\omega\in \Omega$
satisfies the desired properties. Conversely, each family satisfying properties (1)-(4) induces a fixed point of $\mathbb L$ which then  must coincide with $\bar{v}$.
\end{proof}

\begin{proposition}\label{pmeas}
Let $(h_\om^0)_{\om \in \Om}$ be as in Proposition~\ref{prop:uniqueAcim}. Then $h_\om^0$ is a probability measure on $\mathcal B^{1, 1}$  for \paeom.
\end{proposition}

\begin{proof}
Using Lemma~\ref{duallemma}, we have that
 \[
  h_\om^0(\varphi)=\mcl_{\sigma^{-n}\om}^{(n)} h_{\sigma^{-n} \om}^0(\varphi)=h_{\sigma^{-n} \om}^0(\varphi \circ T_{\sigma^{-n} \om}^{(n)}), \quad \text{for $\om \in \Om$ and $\varphi \in C^1$.}
 \]
Hence, using the arguments as in~\cite[Lemma 5.3]{DL}, and equations \eqref{dist_cont} and~\eqref{b}, we find that there exists a constant $D>0$ such that
\[
 \lvert h_\om^0(\varphi)\rvert \le D\lvert \varphi \rvert_\infty \quad \text{for \paeom \ and $\varphi \in C^1$}.
\]
Since $C^1$ is dense in $C^0$, we conclude that for \paeom, \ $h_\om^0$ can be extended to a bounded linear functional on $C^0$. By the Riesz representation theorem,
$h_\om^0$ is a signed measure. By invoking the nonnegativity of $h_\om^0$ together with $h_\om^0(1)=1$, we conclude that $h_\om^0$ is a probability measure for \paeom.
\end{proof}

\begin{proposition}
The top Oseledets space $Y_1(\om)$ of the cocycle $\mcl$ is one-dimensional, and spanned by $h_\om^0$.
\end{proposition}

\begin{proof}
Proposition~\ref{prop:uniqueAcim} and \eqref{wsly_cont} imply that
the top Lyapunov exponent of the cocycle $\mcl$ is equal to $0$. Furthermore, it follows from~\eqref{b} that
 \[
  \lim_{n\to \infty}\frac 1 n \log \lVert \mcl_\om^{(n)} h_\om^0 \rVert_{1,1} =\lim_{n\to \infty}\frac 1 n \log \lVert h_{\sigma^n \om}^0 \rVert_{1,1}\le 0 \quad \text{for \paeom.}
 \]
On the other hand, it follows from~\eqref{dist_cont} that $1=h_\om^0 (1)\le C\lVert h_\om^0\rVert_{1,1}$ for \paeom. Thus,
\[
 \lim_{n\to \infty}\frac 1 n \log \lVert \mcl_\om^{(n)} h_\om^0 \rVert_{1,1} =\lim_{n\to \infty}\frac 1 n \log \lVert h_{\sigma^n \om}^0 \rVert_{1,1}\ge \lim_{n\to \infty}\frac 1 n \log C^{-1}=0
 \quad \text{for \paeom.}
\]
We conclude that $h_\om^0 \in Y_1(\om)$ for \paeom.
We now claim that $h_\om^0$ spans $Y_1(\om)$ for \paeom. Indeed, assume that there exists $g_\om \notin span \{h_\om^0\}$, $g_\om \in Y_1(\om)$ and choose
$\alpha, \beta$ scalars (that depend on $\om$) such that $\lvert \alpha \rvert+\lvert \beta \rvert >0$ and $(\alpha h_\om^0+\beta g_\om)(1)=0$. Then, it follows from~\eqref{DEC_cont} that
\[
 \lim_{n\to \infty}\frac 1 n \log \lVert \mcl_\om^{(n)}(\alpha h_\om^0+\beta g_\om)\rVert_{1,1} \le -\lambda <0.
\]
On the other hand, since $\alpha h_\om^0+\beta g_\om \in Y_1(\om) \setminus \{0\}$ we have
\[
 \lim_{n\to \infty}\frac 1 n \log \lVert \mcl_\om^{(n)}(\alpha h_\om^0+\beta g_\om)\rVert_{1,1}=0,
\]
which yields a contradiction. We conclude that $Y_1(\om)=span \{h_\om^0\}$ and thus $Y_1(\om)$ is one-dimensional for \paeom.

\end{proof}
\section{Quasi-compactness of the twisted cocycle $\mcl^\theta$}
We build a twisted cocycle $\mathcal{L}^\theta$, by setting \[\mcl_\om^{\theta} (h)=\mcl_\om(e^{\theta g(\om, \cdot)} \cdot h), \quad  \text{for $\om \in \Om$, $\theta \in \mathbb C$, and $h\in \mathcal B^{1,1}$.} \]
We will from now write $e^{\theta g(\om, \cdot)} h$ instead of $e^{\theta g(\om, \cdot)} \cdot h$.
Our (centered) observable $g$ will be a map $g\colon \Omega \times X \to \mathbb R$ such that $g(\om, \cdot)\in C^r$ for $\om \in \Om$,
\begin{equation}\label{obs_cont}
 \esssup_{\om \in \Om} \lVert g(\om, \cdot)\rVert_{C^r} <\infty,
\end{equation}
and for \paeom,
\begin{equation}\label{zeromean_cont}
h_\om^0 (g(\om, \cdot))=0.
\end{equation}

This twisted cocycle gives us access to an $\omega$-wise moment-generating function for Birkhoff sums of $g$.
\begin{lemma}
\label{dualtwist}
For $\mathbb{P}$-a.e.\ $\omega\in\Omega$, $h\in \mathcal{B}^{1,1}$ and $\varphi \in C^1(X, \mathbb C)$ one has
\begin{equation}
\label{coding}
  (\mathcal{L}_\omega^{\theta,(n)}h)(\varphi)=h(e^{\theta S_ng(\om, \cdot)}(\varphi \circ T_\omega^n)).
\end{equation}
\end{lemma}
\begin{proof}
One can follow the proof of Lemma 3.3~(part 2)~\cite{DFGTV}, using the definition of the untwisted transfer operator (\ref{aux_cont}) and Lemma \ref{duallemma}.
\end{proof}
The following lemma is required as an auxiliary result in the proof of quasi-compactness of the twisted cocycle (Proposition \ref{twistQC}).
\begin{lemma}\label{tecl}
For $\theta_1, \theta_2 \in B_{\mathbb C}(0, 1):=\{\theta \in \mathbb C: |\theta|<1\}$,
 we have that
 \[
  \esssup_{\om \in \Om} \lVert e^{\theta_1 g(\sigma^{-1} \om, \cdot)}-e^{\theta_2 g(\sigma^{-1} \om, \cdot)}\rVert_{C^2} \le C\lvert \theta_1-\theta_2\rvert.
 \]
\end{lemma}
\begin{proof}
 By applying the mean value theorem for the map $g(z)=e^{z g(\sigma^{-1} \om, x)}$, where $x\in X$ is fixed and using~\eqref{obs_cont}, we find that
 \begin{equation}\label{04}
  \esssup_{\om \in \Om}\lVert e^{\theta_1 g(\sigma^{-1} \om, \cdot)}-e^{\theta_2 g(\sigma^{-1} \om, \cdot)}\rVert_{C^0}\le C\lvert \theta_1-\theta_2\rvert.
 \end{equation}
Furthermore, for $j=1,\ldots,d$
\[
 \begin{split}
  \lVert \partial^j (e^{\theta_1 g(\sigma^{-1} \om, \cdot)}-e^{\theta_2 g(\sigma^{-1} \om, \cdot)})\rVert_{C^0} &=\lVert e^{\theta_1 g(\sigma^{-1} \om, \cdot)}\theta_1
  \partial^j g(\sigma^{-1}\om, \cdot)-e^{\theta_2 g(\sigma^{-1} \om, \cdot)}\theta_2
  \partial^j g(\sigma^{-1}\om, \cdot)\rVert_{C^0} \\
  &\le \lvert \theta_1-\theta_2\rvert \cdot \lVert e^{\theta_1 g(\sigma^{-1} \om, \cdot)}
  \partial^j g(\sigma^{-1}\om, \cdot)\rVert_{C^0} \\
  &\phantom{\le}+\lvert \theta_2\rvert \cdot \lVert e^{\theta_1 g(\sigma^{-1} \om, \cdot)}-e^{\theta_2 g(\sigma^{-1} \om, \cdot)}\rVert_{C^0}\cdot \lVert \partial^j g(\sigma^{-1}\om, \cdot)\rVert_{C^0}
 .\end{split}
\]
It now follows from~\eqref{obs_cont} and~\eqref{04} that
\[
 \esssup_{\om \in \Om} \lVert \partial^j (e^{\theta_1 g(\sigma^{-1} \om, \cdot)}-e^{\theta_2 g(\sigma^{-1} \om, \cdot)})\rVert_{C^0}\le C\lvert \theta_1-\theta_2\rvert.
\]
One can now proceed and obtain the same estimates for the second derivatives of the map $e^{\theta_1 g(\sigma^{-1} \om, \cdot)}-e^{\theta_2 g(\sigma^{-1} \om, \cdot)}$ which implies the desired
conclusion.
\end{proof}
We need the following basic regularity result for the operators $\mathcal{L}^\theta_\omega$.
\begin{proposition}
\label{twistcty}
 There exists a continuous function $K\colon \mathbb C \to (0, \infty)$ such that
 \begin{equation}\label{ubt_cont}
  \lVert \mcl_\om^{\theta}h\rVert_{1,1} \le K(\theta)\lVert h\rVert_{1,1}, \quad \text{for $h\in \mathcal B^{1,1}$, $\theta \in \mathbb C$ and \paeom.}
 \end{equation}
\end{proposition}

\begin{proof}
 We first note that it follows from~\eqref{ub_cont} that
 \[
  \lVert \mcl_\om^{\theta}h\rVert_{1,1}=\lVert \mcl_\om(e^{\theta g(\om, \cdot)}h)\rVert_{1,1} \le K \lVert e^{\theta g(\om, \cdot)}h\rVert_{1,1}, \quad \text{for $h\in \mathcal B^{1,1}$, $\theta \in \mathbb C$ and \paeom.}
 \]
Hence, we need to estimate $\lVert e^{\theta g(\om, \cdot)}h\rVert_{1,1}$. Note that by~\eqref{norms},
\[
 \lVert e^{\theta g(\om, \cdot)}h\rVert_{1,1}=\max \{ \lVert e^{\theta g(\om, \cdot)}h\rVert^{\sim}_{0,1}, \lVert e^{\theta g(\om, \cdot)}h\rVert^{\sim}_{1,2}\}.
\]
It follows easily from~\eqref{auxnorms} that
\[
 \lVert e^{\theta g(\om, \cdot)}h\rVert^{\sim}_{0,1} \le \big{(}\max_{1\le i\le N}\sup_{\substack{\chi \colon \overline{B}(x, A\delta)\to \mathbb R^{d_u} \\ \chi \in \Xi_i}}\lVert (e^{\theta g(\om, \cdot)}\circ \psi_i) \circ (\Id, \chi)\rVert_{C^1}\big{)}\cdot
 \lVert h\rVert^{\sim}_{0,1}
\]
and
\[
\begin{split}
 \lVert e^{\theta g(\om, \cdot)}h\rVert^{\sim}_{1,2} &\le \big{(}\max_{1\le i\le N}\sup_{\substack{\chi \colon \overline{B}(x, A\delta)\to \mathbb R^{d_u} \\ \chi \in \Xi_i}}\lVert (e^{\theta g(\om, \cdot)}\circ \psi_i) \circ (\Id, \chi)\rVert_{C^2}\big{)}\cdot
 \lVert h\rVert^\sim_{1,2}\\
 &\phantom{\le} + \bigg{(} \max_{\substack{1\le j\le d \\ 1\le i \le N}}\sup_{\substack{\chi \colon \overline{B}(x, A\delta)\to \mathbb R^{d_u} \\ \chi \in \Xi_i}}\lVert \big{[}\partial^j(e^{\theta g(\om, \cdot)}\circ \psi_i)]\circ (\Id, \chi)\rVert_{C^1}
 \bigg{)}\cdot \lVert h\rVert^\sim_{0,1},
 \end{split}
\]
which together with~\eqref{obs_cont} implies the desired conclusion.
\end{proof}

We can now state the main result of this section on quasi-compactness.
\begin{proposition}
\label{twistQC}
 For $\theta$ close to $0$, the cocycle $(\mcl_\om^\theta)_{\om \in \Om}$ is quasi-compact.
\end{proposition}

\begin{proof}
 We follow closely~\cite[Lemma 3.13]{DFGTV}. Observe~\eqref{wsly_cont} and choose $N\in \mathbb N$ such that $\gamma:=Ba^N<1$. Hence,
 \[
 \begin{split}
  \lVert \mcl_\om^{\theta, (N)} h\rVert_{1,1} &\le \lVert  \mcl_\om^{ (N)} h\rVert_{1,1}+\lVert \mcl_\om^{\theta, (N)}-\mcl_\om^{(N)}\rVert_{1,1} \cdot \lVert h\rVert_{1,1} \\
  &\le
  \gamma \lVert h\rVert_{1,1}+B\lVert h\rVert_{0,2}+\lVert \mcl_\om^{\theta, (N)}-\mcl_\om^{(N)}\rVert_{1,1} \cdot \lVert h\rVert_{1,1}.
  \end{split}
 \]
 On the other hand, we have that
\[
 \mcl_\om^{\theta, (N)}-\mcl_\om^{(N)}=\sum_{j=0}^{N-1} \mcl_{\sigma^{N-j} \om}^{\theta, (j)}(\mcl_{\sigma^{N-1-j} \om}^\theta -\mcl_{\sigma^{N-1-j} \om})\mcl_\om^{(N-1-j)}.
\]
It follows from~\eqref{ub_cont} and~\eqref{ubt_cont} that
\[
 \lVert \mcl_\om^{(N-1-j)}\rVert_{1,1} \le K^{N-1-j} \quad \text{and} \quad \lVert \mcl_{\sigma^{N-j} \om}^{\theta, (j)}\rVert_{1,1} \le K(\theta)^j.
\]
Furthermore, using~\eqref{ub_cont}, we have that for any $h\in \mathcal B^{1,1}$ and \paeom,
\[
 \lVert (\mcl_\om^\theta -\mcl _\om)(h) \rVert_{1,1}=\lVert \mcl_\om (e^{\theta g(\om, \cdot)}h-h)\rVert_{1,1} \le K\lVert (e^{\theta g(\om, \cdot)}-1)h\rVert_{1,1}.
 \]
Moreover,
\[
 \lVert (e^{\theta g(\om, \cdot)}-1)h\rVert_{1,1}=\max \{ \lVert (e^{\theta g(\om, \cdot)}-1)h\rVert^\sim_{0,1}, \lVert (e^{\theta g(\om, \cdot)}-1)h\rVert^\sim_{1,2} \}.
\]
Now Lemma~\ref{tecl} (applied for $\theta_1=\theta$ and $\theta_2=0$) implies that there exists $C>0$ such for $\theta \in B_{\mathbb C}(0, 1)$,
\[
  \lVert (e^{\theta g(\om, \cdot)}-1)h\rVert_{1,1} \le C \lvert \theta \rvert \lVert h\rVert_{1,1} \quad \text{for $h\in \mathcal B^{1,1}$.}
\]
We conclude that
\[
 \lVert \mcl_\om^{\theta, (N)}-\mcl_\om^{(N)}\rVert_{1,1} \le C\lvert \theta \rvert \sum_{j=0}^{N-1}K^{N-1-j}K(\theta)^j,
\]
and therefore there exists $\tilde \gamma \in (0, 1)$ such that for any $\theta$ sufficiently close to $0$ and $h\in \mathcal B^{1,1}$,
\begin{equation}\label{0320_cont}
 \lVert \mcl_\om^{\theta, (N)} h\rVert_{1,1} \le \tilde \gamma \lVert h\rVert_{1,1}+B\lVert h\rVert_{0,2}.
\end{equation}
Similarly, one can show that there exists $\tilde B>0$ such that for any $\theta$ sufficiently close to $0$ and $h\in \mathcal B^{1,1}$,
\begin{equation}\label{0321_cont}
 \lVert \mcl_\om^\theta h\rVert_{0,2}\le \tilde B\lVert h\rVert_{0,2}.
\end{equation}
The conclusion of the proposition follows from~\eqref{0320_cont} and~\eqref{0321_cont} by arguing as in the quasi-compactness part of the proof of~\cite[Theorem 3.12]{DFGTV}.
\end{proof}

\section{Regularity of the top Oseledets space of the twisted cocycle}
\label{sec:regularity}
Let $\mathcal S'$ be the space of measurable maps $\mathcal V\colon \Omega \to \mathcal B^{1,1}$ with the property that
\[\lVert \mathcal V\rVert_\infty:=\esssup_{\om \in \Om}\lVert \mathcal V(\om)\rVert_{1,1}<\infty.\] Then, $(\mathcal S', \lVert \cdot \rVert_\infty)$ is a Banach space.
Furthermore, let $\mathcal S$ be the set of all $\mathcal V\in \mathcal S'$ such that $\mathcal V(\om)(1)=0$ for \paeom.
Arguing as in the proof of  Proposition~\ref{prop:uniqueAcim}, it is easy to verify that $\mathcal S$ is  a closed subspace of
$\mathcal S'$.
For $\mathcal V\in \mathcal S'$ and $\omega \in \Omega$ we will often write $\mathcal V_\omega$ instead of $\mathcal V(\omega)$.

\subsection{Regularity of the cocycles}
\begin{lemma}\label{lem:analyt}\quad
\begin{enumerate}
\item  \label{it:analyt1}
 For $\mathbb{P}$-a.e. $\omega\in\Omega$, the map $\theta \mapsto \mcl_\om^\theta$ is analytic in the norm topology of $\mathcal B^{1,1}$.
 \item \label{it:analyt2}
  The map
 $\mc P:  B_{\C}(0, 1) \times \mc S\to \mc S$, given by
  $\mc P (\theta, \mc V)_\om = \mcl_{\sig^{-1}\om}^\theta (\mc V_{\sig^{-1}\om})$ is analytic in $\theta$ and bounded, linear in $\mc V$.
  In particular, $\mc P$ is $C^\infty$.
   \item \label{it:analyt3}
  The map
 $\mc P_1:  B_{\C}(0, 1) \times \mc S\to \mc L^\infty(\Om)$, given by
  $\mc P_1 (\theta, \mc V)_\om =  (\mcl_{\sig^{-1}\om}^\theta (\mc V_{\sig^{-1}\om}))( 1) $ is analytic in $\theta$ and bounded, linear in $\mc V$.   In particular, $\mc P_1$ is $C^\infty$.
  \end{enumerate}
\end{lemma}
\begin{proof}
We  claim that for every $h \in \mc B^{1,1}$, the following holds:
\begin{equation}\label{eq:analytic}
\mcl_\om^\theta(h) = \sum_{k=0}^\infty \frac{\theta^k}{k!} \mcl_\om(g(\om, \cdot)^k h), \quad \text { in } \mc B^{1,1}.
\end{equation}
To verify this, note that \cite[Lemma 3.2]{GL} implies that \[ \| \mcl_\om(g(\om, \cdot)^k h) \|_{{1,1}} \leq C \| g(\om, \cdot)^k\|_{C^2} \|h\|_{{1,1}}\leq C \| g(\om, \cdot)\|^k_{C^2} \|h\|_{{1,1}},\] so by \eqref{obs_cont}, the RHS of \eqref{eq:analytic} is a well defined element of $\mc B^{1,1}$.
The fact that it coincides with $\mcl_\om^\theta(h)$ is straightforward to check, using linearity of $\mcl_\om$, the power series expansion of $e^{\theta g(\om, \cdot)}$, and testing against functions $\varphi \in \mc C^1$. This concludes the proof of Lemma~\ref{lem:analyt}\eqref{it:analyt1}.

Let us prove Lemma~\ref{lem:analyt}\eqref{it:analyt2}.
For each $k\geq 0$ and $\mc V \in \mc S$, let $(g^k \cdot \mc V) (\om, \cdot) := g(\om, \cdot)^k \mc V (\om, \cdot)$. Then,  $g^k \cdot \mc V \in \mc S$, because of \eqref{obs_cont} and \cite[Lemma 3.2]{GL}.
We claim that
\begin{equation}\label{eq:analyticInS0}
\mc P (\theta, \mc V) = \sum_{k=0}^\infty \frac{\theta^k}{k!} \mc P (0, g^k \cdot \mc V)
 \quad \text { in } \mc S.
\end{equation}
Indeed, \eqref{eq:analytic} implies that
\begin{equation}\label{eq:analyticPW}
\mc P (\theta, \mc V)_\om = \sum_{k=0}^\infty \frac{\theta^k}{k!} \mc P (0, g^k \cdot \mc V)_\om
 \quad \text { in } \mc B^{1,1}.
\end{equation}
Furthermore, using once again  \cite[Lemma 3.2]{GL}, in combination with the uniform over $\om$ bounds \eqref{ub_cont} and \eqref{obs_cont}, we have that there exists $C> 0$ such that for \paeom,
\begin{equation}\label{eq:bdPowerSeries}
\| \mc P (\theta, \mc V)_\om \|_{1,1} \leq \sum_{k=0}^\infty \frac{\theta^k}{k!} \|\mc P (0, g^k \cdot \mc V)_\om\|_{1,1} \leq
C \sum_{k=0}^\infty \frac{\theta^k}{k!}
 \esssup_{\om \in \Om} \lVert g(\om, \cdot)\rVert_{C^2}^k \|\mc V\|_\infty.
\end{equation}
Hence, the series in \eqref{eq:analyticInS0} indeed converges in $\mc S$ and yields  analyticity as required.
The fact that $\mc{V} \mapsto \mc P(\theta, \mc V)$, and also  $\mc{V} \mapsto \mc P (0, g^k \cdot \mc V)$ is linear and bounded is straightforward to check.
Hence, the $C^\infty$ claim follows immediately.

The proof of  Lemma~\ref{lem:analyt}\eqref{it:analyt3} is similar to that of Lemma~\ref{lem:analyt}\eqref{it:analyt2}. Indeed,
\begin{equation}\label{eq:analyticInS}
\mc P_1 (\theta, \mc V)_\om = \sum_{k=0}^\infty \frac{\theta^k}{k!} \langle \mc P (0, g^k \cdot \mc V)_\om , 1 \rangle,
\end{equation}
and \eqref{dist_cont} implies that $|\langle \mc P (0, g^k \cdot \mc V)_\om , 1 \rangle| \leq C \| \mc P (0, g^k \cdot \mc V)_\om \|_{1,1}$, which was bounded uniformly over $\om$ in \eqref{eq:bdPowerSeries}.
Hence, the series \eqref{eq:analyticInS} converges to $\mc P_1 (\theta, \mc V)$ in $L^\infty(\Om)$.
\end{proof}

\subsection{An auxiliary function $F$ and its regularity}
\label{RoF}
For $\theta \in \C$ and $\mc W\in \mathcal S$,
set
\begin{equation}\label{defF}
 F(\theta, \mc W)(\om)=\frac{\mcl_{\sigma^{-1}\om}^\theta (\mc W(\sigma^{-1}\om)+h_{\sigma^{-1} \om}^0)}{\mcl_{\sigma^{-1}\om}^\theta (\mc W(\sigma^{-1}\om)+h_{\sigma^{-1} \om}^0)(1)}
 -\mc W(\om )-h_\om^0, \quad \om \in \Om.
\end{equation}
We define two further auxiliary functions, which will be used in the sequel. Let $G \colon \C \times \mc S\to \mc S'$ and $H \colon \C \times \mc S \to L^\infty(\Omega)$ be given by
\begin{equation}
\label{fG}
 G(\theta, \mc W)(\om):=\mc P(\theta, \mc W+ h^0) (\om)= \mcl_{\sigma^{-1}\om}^\theta (\mc W_{\sigma^{-1}\om}+h_{\sigma^{-1} \om}^0)
 , \quad \om \in \Om,
\end{equation}
\begin{equation}\label{fH}
 H(\theta, \mc W)(\om):= \mc P_1(\theta, \mc W+ h^0) (\om)=\mcl_{\sigma^{-1}\om}^\theta (\mc W_{\sigma^{-1}\om}+h_{\sigma^{-1} \om}^0)(1), \quad \om \in \Om.
\end{equation}
It follows readily from \eqref{b} and Lemma~\ref{lem:analyt} that $G$ and $H$ are  well defined, and in fact $C^\infty$ functions.
Direct calculations, analogous to those of \cite[Appendix B]{DFGTV}, yield the following:

\begin{lemma}\label{lem:derivGH}
For $\om \in \Om; \theta, z \in \C$; $\mc W, \mc H \in \mathcal S$, the following identities hold:
 \begin{align}
\label{D1G}
 D_1 G(\theta, \mc W)(z)_\om &=z \mathcal  L_{\sigma^{-1} \om}(g(\sigma^{-1}\om ,\cdot)e^{\theta g(\sigma^{-1} \om ,\cdot)}(\mc W_{\sigma^{-1} \om}+h^0_{\sigma^{-1} \om})),
\\
\label{D2G}
   D_2G(\theta, \mc{W})(\mc{H})_\om&=\mathcal L^\theta_{\sigma^{-1} \om}(\mc{H}_{\sigma^{-1} \om}),
   \\
   \label{D11G}
 D_{11} G(\theta, \mc W)(z_1, z_2)_\om&=z_1 z_2 \mathcal  L_{\sigma^{-1} \om}(g(\sigma^{-1}\om ,\cdot)^2e^{\theta g(\sigma^{-1} \om ,\cdot)}(\mc W_{\sigma^{-1} \om}+h^0_{\sigma^{-1} \om})),
\\
\label{D12G}
 D_{12}G(\theta, \mc W)(z,\mc H)_\om&=D_{21}G(\theta, \mc W)(\mc H,z)_\om=z \mathcal L_{\sigma^{-1} \om}(g(\sigma^{-1} \om, \cdot)e^{\theta g(\sigma^{-1} \om, \cdot)} \mc H_{\sigma^{-1} \om}),
\\
\label{D22G}
D_{22}G &=0.
\end{align}
Moreover, the expressions for the derivatives of $H$ are equal to the corresponding expression for $G$ applied to the constant function 1.
\end{lemma}

\begin{lemma}\label{lem:FwellDef}
 There exist  $\ep, R>0$ such that
  $F \colon \mc{D} \to \mc{S}$ is a well-defined map on \[ \mc{D}:=\{ \theta \in \C : |\theta|<\ep \} \times B_{\mc{S}}(0,R),\] where $B_{\mc{S}}(0,R)$ denotes the ball of radius $R$ in $\mc{S}$ centered at $0$.
\end{lemma}

\begin{proof}
 Let $G$ and $H$ be defined as in (\ref{fG}) and (\ref{fH}).
 The function $H$ is continuous on a neighborhood  of $(0, 0)$ in $\mathbb C \times \mc S$ and obviously $H(0, 0)(\om)=h_\om^0(1)=1$ for \paeom. Hence,
 \[
  \lvert H(\theta, \mc W)(\om)\rvert \ge 1-\lvert H(0, 0)(\om)-H(\theta, \mc W)(\om)\rvert \ge 1-\lVert H(0, 0)-H(\theta, \mc W)\rVert_{L^\infty},
 \]
for \paeom. Continuity of $H$ implies that $\lVert H(0, 0)-H(\theta, \mc W)\rVert_{L^\infty}\le \frac 1 2$ for all $(\theta, \mc W)$ in a neighborhood of $(0, 0)$ in $\mathbb C \times \mc S$
and hence, in such a neighborhood,
\[
 \essinf_{\om \in \Om} \lvert H(\theta, \mc W)(\om)\rvert \ge \frac 1 2.
\]
This together with~\eqref{b} and a simple observation that $F(\theta, \mc W)(1)=0$  immediately yields the desired conclusion.

\end{proof}

Notice that map $F$ defined by~\eqref{defF} satisfies $F(\theta, \mc W)(\om)= G(\theta, \mc W)(\om)/H (\theta, \mc W)(\om) - \mc W(\om )-h_\om^0$. The proof of Lemma~\ref{lem:FwellDef} ensures   that for $(\theta, \mc W)$ in a neighbourhood $\mathcal{D}$ of $(0,0)\in \C \times \mc{S}$,
$\essinf_{\om \in \Om} \lvert H(\theta, \mc W)(\om)\rvert \ge \frac 1 2$.
Thus, the following result is a direct consequence of  Lemma~\ref{lem:derivGH}.

\begin{proposition}\label{difF}
The  map $F$ defined by~\eqref{defF} is of class $C^\infty$ on the neighborhood $\mathcal{D}$ of $(0, 0)\in \mathbb C \times \mc{S}$ from Lemma~\ref{lem:FwellDef}. Moreover, for $\om \in \Om, (\theta, \mc W) \in \mathcal{D}$  and $\mc H \in \mathcal S$,
 \[
  D_2 F(\theta, \mc W) (\mc H)_\om=\frac{1}{H(\theta, \mc W)(\om)}\mathcal L_{\sigma^{-1} \om}^\theta \mc H_{\sigma^{-1} \om}-\frac{\mathcal L_{\sigma^{-1} \om}^\theta
  \mc H_{\sigma^{-1} \om}(1)}{[H(\theta, \mc W)(\om)]^2}G(\theta, \mc W)_\om-\mc H_\om,
 \]
\[
 \begin{split}
 D_1 F(\theta, \mc W)_\om &=\frac{1}{H(\theta, \mc W)(\om)}\mathcal L_{\sigma^{-1} \om}(g(\sigma^{-1} \om, \cdot)e^{\theta  g(\sigma^{-1} \om, \cdot)} (\mc W_{\sigma^{-1} \om}+
 h_{\sigma^{-1} \om}^0)) \\
 &\phantom{=}-\frac{\mcl_{\sigma^{-1} \om}( g(\sigma^{-1} \om, \cdot)e^{\theta  g(\sigma^{-1} \om, \cdot)} (\mc W_{\sigma^{-1} \om}+
 h_{\sigma^{-1} \om}^0))(1)}{[H(\theta, \mc W)(\om)]^2}\mathcal L_{\sigma^{-1} \om}^\theta (\mc W_{\sigma^{-1} \om}+h_{\sigma^{-1} \om}^0),
 \end{split}
\]
where we have identified $D_1 F(\theta, \mc W)$ with its value at $1$.

\end{proposition}

\begin{lemma}\label{thm:IFT}
Let $\mc{D}=\{ \theta \in \C : |\theta|<\ep \} \times B_{\mc{S}}(0,R)$ be as in Lemma~\ref{lem:FwellDef}. Then,
$F:\mc{D} \to \mc{S}$ is $C^\infty$  and
the equation
\begin{equation}
F(\theta, \mc{W})=0
\end{equation}
has a unique solution $O(\theta) \in \mc{S}$, for every $\theta$ in a neighborhood  of 0.
Furthermore, $O(\theta)$ is a $C^\infty$ function of $\theta$.
\end{lemma}

\begin{proof}
 Note that $F(0, 0)=0$. Furthermore, Proposition~\ref{difF} implies that $F$ is of class $C^\infty$ on a neighborhood of $(0, 0)$. In addition, Lemma~\ref{lem:derivGH} implies that
 \[
  (D_2F(0, 0)\mathcal X)(\om)=\mcl_{\sigma^{-1} \om}\mathcal X(\sigma^{-1} \om)-\mathcal X(\om), \quad \om \in \Om, \ \mathcal X\in \mc S.
 \]
 Using~\eqref{DEC_cont} and proceeding as in~\cite[Lemma 3.5]{DFGTV}, one can show that $D_2F(0, 0)$ is invertible and that
 \begin{equation}\label{d2f}
  (D_2F(0, 0)^{-1}\mathcal X)(\om)=-\sum_{j=0}^\infty \mcl_{\sigma^{-j} \om}^{(j)}\mathcal X(\sigma^{-j} \om) \quad \om \in \Om, \ \mathcal X\in S.
 \end{equation}
The conclusion of the lemma now follows directly  from the implicit function theorem.
\end{proof}

\section{Properties of $\Lambda(\theta)$}
\label{sec:Lambda}
Let $0<\ep<1$ be as in Lemma~\ref{lem:FwellDef} and $O(\theta)$ be as in Lemma~\ref{thm:IFT}.
Let
\begin{equation}\label{eq:vomt}
h_\om^\theta:= h_\om^0 +O(\theta)(\om) \in \mathcal B^{1,1}, \quad \om \in \Om.
\end{equation}
We notice that $h_\om^\theta(1)=1$ and by Lemma~\ref{thm:IFT}, $\theta \mapsto h^\theta$ is continuously differentiable.

Let us define
\begin{equation}\label{eq:hatLam}
 \hat\Lambda (\theta) :=  \int \log \Big|h_\om^\theta (e^{\theta g(\om, \cdot)})\Big|\, d\bbp(\om),
\end{equation}
and
\begin{equation}\label{eq:int}
\lot :=  h_\om^\theta (e^{\theta g(\om, \cdot)})
=  \mcl_\om^\theta h_\om^\theta (1).
\end{equation}

\subsection{A differentiable lower bound for $\Lambda(\theta)$}
Lemma \ref{Lamhatlemma} deals with differentiability properties of  $\hat\Lambda (\theta)$.

\begin{lemma}
\label{Lamhatlemma}\quad
\begin{enumerate}
\item
For every $\theta \in B_\C(0,\ep)$, $ \hat\Lambda (\theta)\leq \Lambda (\theta)$.
\item
$\hat\Lambda$ is differentiable on a neighborhood of 0,
and
\[
\hat \Lambda' (\theta)=
\Re \Bigg( \int \frac{ \overline{\lot}   ( (O(\theta)(\om)+h_\om^0 )(g(\om, \cdot)e^{\theta g(\om, \cdot)})+O'(\theta)(\om)(e^{\theta g(\om, \cdot)}) )}{|\lot |^2}\, d\bbp(\om) \Bigg),
\]
where $\Re (z)$ denotes the real part of $z$ and $\overline{z}$ the complex conjugate of $z$.
\item
For \paeom, and $\theta$ in a neighborhood of 0, the map $\theta \mapsto Z_\om(\theta):=Z(\theta, \omega)$ is differentiable. Moreover,
\[
 Z_\om'(\theta)=\frac{\Re \Big( \overline{\lot}   ( (O(\theta)(\om)+h_\om^0 )(g(\om, \cdot)e^{\theta g(\om, \cdot)})+O'(\theta)(\om)(e^{\theta g(\om, \cdot)}) )\Big)}{| \lot|^2}.
\]

\item
$\hat \Lambda'(0)=0$.
 \end{enumerate}
 \end{lemma}

\begin{proof}
The proof of part 1 is identical to the proof of Lemma 3.8 \cite{DFGTV} replacing $\|\cdot\|_{\mathcal{B}}$ with $\|\cdot\|_{1,1}$ and $\|\mathcal{L}_\omega^{\theta,(n)}v^\theta_\omega\|_1$ with $|\mathcal{L}_\omega^{\theta,(n)}h^\theta_\omega(1)|$.

The proof of part 2 is identical to the proof of Lemma 3.9 \cite{DFGTV}, using Lemma \ref{thm:IFT} in place of Lemma 3.5 \cite{DFGTV} and replacing the final two equation blocks with:
\[
 \begin{split}
  \lvert (O(\theta)(\om)+h_\om^0 )(g(\om, \cdot)e^{\theta g(\om, \cdot)})\rvert & \le C\lVert O(\theta)(\om)+h_\om^0 \rVert_{1,1}\cdot \lVert g(\om, \cdot)e^{\theta g(\om, \cdot)}\rVert_{C^1}
 \\&\le C\lVert O(\theta)\rVert_\infty+C,
 \end{split}
\]
and
\[
 \lvert O'(\theta)(\om)(e^{\theta g(\om, \cdot)})\rvert \le C\lVert O'(\theta)(\om)\rVert_{1,1} \cdot \lVert e^{\theta g(\om, \cdot)}\rVert_{C^1} \le
C \lVert O'(\theta)\rVert_\infty.
\]

The proof of part 3 is identical to the proof of Lemma 3.10 \cite{DFGTV}, using differentiability of $H$ and $O$ in Lemmas \ref{lem:derivGH} and \ref{thm:IFT}.

The proof of part 4 is identical to proof of Lemma 3.11 \cite{DFGTV}.
\end{proof}

\subsection{One-dimensionality of $Y_1^\theta(\om)$ and differentiability of $\Lambda$}

Let $Y_1^\theta(\om)$ denote the top Oseledets subspace of the cocycle $(\mcl_\om^\theta)_{\om \in \Om}$. The proof of part 1 of the following result can be obtained by
repeating the argument as in~\cite[Theorem 3.12]{DFGTV}, using Proposition \ref{twistQC}.
Part 2 follows by arguing as in~\cite[Corollary 3.14]{DFGTV}.
\begin{proposition}\label{cor:LamHatLam}For $\theta\in \C$ near 0
\begin{enumerate}
\item
 $\dim Y_1^\theta(\om)=1$.
\item
 $\Lam(\theta)=\hat\Lam(\theta)$.
In particular, $\Lam(\theta)$ is differentiable near $0$ and $\Lam'(0)=0$.
\end{enumerate}
\end{proposition}

\subsection{Convexity of $\Lambda(\theta)$}
\label{sec:convexity}
By Proposition~\ref{pmeas}, we can regard $h_\om^0$ as Borel probability measure on $X$ which we will denote by $\mu_\om$. The family $(\mu_\om)_{\om \in \Om}$ induces a probability measure on $\Om \times X$ given by
\[
 \mu(A\times B)=\int_A \mu_\om (B)\, d\mathbb P(\om), \quad \text{for measurable sets $A\subset \Om$ and $B\subset X$.}
\]
Then, $\mu$ is invariant for the skew-product transformation $\tau \colon \Om \times X\to \Om \times X$ defined by
\[
 \tau(\om, x)=(\sigma \om, T_\om (x)), \quad \om \in \Om, \ x\in X.
\]
Obviously, the variance $\Sig^2$ defined in (\ref{variance}) is nonnegative.
From now on we shall assume that $\Sig^2>0$. Otherwise we can invoke and adapt the coboundary case proved  in  Proposition 3 in our paper \cite{DFGTV2}, which says that $\Sig^2=0$ if and only if there exists $r\in L^2_{\mu}(\Omega \times X)$ such that $g=r-r\circ \tau.$

\begin{proposition}\label{lem:Lam''0}
On a neighbourhood of 0,
\begin{enumerate}
\item $\Lambda$ is of class $C^2$  and  $\Lambda''(0)=\Sig^2$.
\item $\Lambda$ is strictly convex.
\end{enumerate}
\end{proposition}

\begin{proof}
The proof of part 1 is identical to the proof of Lemma 3.15 \cite{DFGTV} and part (ii) is a direct consequence of part 2.
\end{proof}

\section{Large deviation principle and central limit theorem}
\label{sec:ldp_clt}

For $\theta \in \C$ sufficiently close to $0$, we have that $\dim Y_1^\theta (\om)=1$. Choose $h_\om^\theta \in Y_1^\theta (\om)$ such that $h_\om^\theta(1)=1$. We note that
$h_\om^\theta$ is actually given by~\eqref{eq:vomt}. Furthermore, let $\lambda_\om^\theta \in \mathbb C$ be such that
\begin{equation}
\label{eq:def_lambdas}
\mathcal{L}^{\theta}_\omega h^\theta_\omega=\lambda^\theta_\omega h^\theta_{\sigma\omega}.
\end{equation}
Note that
\begin{equation}\label{eq:lam}
\lambda^\theta_\omega= h_\om^\theta (e^{\theta g(\om, \cdot)}),
\end{equation}
which coincides with~\eqref{eq:int}.
Next, let us fix $\phi^\theta_\omega \in Y^{*\,\theta}_\om$ so that $\phi^\theta_\omega(h^\theta_\omega)=1$. Furthermore, one can show (see~\cite[p. 30]{DFGTV}) that
\begin{equation} \label{eq:DualEig}
(\mathcal{L}^{\theta}_\omega)^*\phi^\theta_{\sigma\omega}=\lambda^\theta_\omega \phi^\theta_{\omega}.
\end{equation}
\begin{remark}
\label{phidiffrem}
The differentiability of $\theta\mapsto\phi^\theta$ follows similarly to the presentation in \cite[Appendix C]{DFGTV}. The proofs of Lemmas C.4 and C.6 make use of regularity estimates (90) and (99) in terms of variation;  in the present work, these estimates may be replaced with $C^1$ estimates.  In the proof of Lemma C.2, the expression $\|v_\om^0\|_1$ may be replaced with $\lvert h_\om^0(1)\rvert$ and bounded by (\ref{dist_cont}) in the present work.
\end{remark}
In addition, let
\[
 \mathcal B^{1,1}=Y_\om^\theta \oplus H_\om^\theta \quad \text{and} \quad (\mathcal B^{1,1})^*=Y^{*\, \theta}_\om \oplus H^{*\, \theta}_\om
\]
be the Oseledets splitting of cocycles $(\mcl_\om^\theta)_{\om \in \Om}$ and $((\mcl_\om^{\theta})^*)_{\om \in \Om}$ respectively into a direct sum of the top space and the sum of all other
Oseledets subspaces.

\subsection{Large deviation principle}

The following lemmas link the limits of characteristic functions of Birkhoff sums to the function $\Lambda$.
 \begin{lemma}\label{L:growthExpSums}
Let $\theta\in \C$ be sufficiently close to $0$ and
 $h\in \mathcal B^{1,1}$ be such that $h\notin H_\om^\theta$, i.e. $\phi^\theta_\omega (h) \neq 0$.
Then,
\[
\lim_{n\to\infty}\frac{1}{n} \log \Big| h(e^{\theta S_n g(\om, \cdot)}) \Big| =  \Lam(\theta).
\]
\end{lemma}

\begin{proof}
Identical to the proof of Lemma 4.2~\cite{DFGTV} with $h\in \mathcal{B}^{1,1}$ in the present paper playing the role of $\int f\cdot\ dm$ in the proof of Lemma 4.2~\cite{DFGTV}, and Lemma \ref{coding} replacing (43)~\cite{DFGTV}.
 \end{proof}

\begin{lemma}\label{need}
For all complex $\theta$ in a neighborhood of 0, and \paeom, we have that
\[
 \lim_{n\to \infty} \frac 1 n \log  \Big|\int e^{\theta S_n g(\om, x)} \, d\mu_\om(x) \Big|=\Lambda (\theta).
\]
\end{lemma}

\begin{proof}
We follow the proof of Lemma 4.3 \cite{DFGTV}, observing that
 \[
  \int e^{\theta S_n g(\om, x)} \, d\mu_\om(x)=h_\om^0(e^{\theta S_n g(\om, \cdot)}),
 \]
 and recalling the differentiability of the map $\theta \mapsto \phi^\theta$ in Remark \ref{phidiffrem}.
\end{proof}

\begin{proof}[Proof of Theorem \ref{thm:ldt}]
Following the proof of Theorem A \cite{DFGTV}, by applying Proposition~\ref{lem:Lam''0} and Lemma~\ref{need}, together with the G\"artner-Ellis theorem (\cite{hennion} and Theorem 4.1 \cite{DFGTV}), we obtain the large deviation principle.
\end{proof}

\subsection{Central limit theorem}
The proof of the following result is completely analogous to the proof of~\cite[Lemma 4.4]{DFGTV}.
\begin{lemma}\label{lem:UnifExpDecayY2}
There exist $C>0, 0<r<1$ such that for every $\theta \in \mathbb{C}$ sufficiently close to 0, every $n\in \N$ and $\paeom$, we have
\begin{equation}
\Big|  \mathcal L_\om^{\theta, (n)}(h_\om^0 -\phi_\om^{\theta}(h_\om^0) h_{\om}^{\theta}) (1) \Big|
\leq Cr^n .
\end{equation}
\end{lemma}

\begin{proof}[Proof of Theorem \ref{thm:clt}]
The proof is identical to the proof of Theorem B \cite{DFGTV}, with the same modifications as those listed in the proof of Lemma \ref{L:growthExpSums}.
Differentiability of $\theta\mapsto\phi^\theta$ is used (see Remark \ref{phidiffrem}) as well as Lemma \ref{dualtwist} to obtain the coding of the Birkhoff sums via the twisted transfer operator.  Lemma 4.5 \cite{DFGTV} is proved is proved in an identical way.
\end{proof}

\section{Local central limit theorem}
We begin by recalling the concept of $\mathbb P$-continuity which we will also use in section 9.2.1.   We say that  our cocycle is $\mathbb P$-continuous (a concept introduced in \cite{thieullen}) if   the map $\omega \mapsto T_\omega$ has  $\mathbb P$-a.e. a countable range (besides being measurable).
This implies that $\omega\mapsto \mathcal{L}_\omega$  is continuous on each of  countably many Borel subsets of $\Omega$, whose union has full $\mathbb{P}$ measure. We refer to~\cite{FLQ2} for details.

In our earlier paper~\cite{DFGTV} we proved  the local central limit theorem in the non-arithmetic case (we also separately treated the arithmetic case) under the condition that  we called (L) in the introduction, namely
\begin{itemize}
\item (L) For $\mathbb{P}$-a.e. $\omega\in \Omega$ and for every compact interval $J\subset \mathbb{R}\backslash{\{0\}}$ there exists $C=C(\omega)>0$ and $\rho\in (0,1)$ such that
    \begin{equation}\label{LD}
    ||\mathcal{L}^{it, (n)}||_{\mathcal{B}}\le C \rho^n, \ \text{for} \ t\in J\ \text{and}\  n\ge 0.
   \end{equation}
   \end{itemize}
   Moreover
 under the assumption that the cocycle is $\mathbb P$-continuous,  we proved~\cite[Lemma 4.7.]{DFGTV}  that (L) is equivalent to the  following {\em aperiodicity}  condition
\begin{itemize}

        \item  For every $t\in \mathbb{R}$, either $\Lambda(it)<0$ or the cocycle $\mathcal{L}_{\omega}^{it}$ is quasicompact and the equation
            $$
            e^{itg(\omega,x)}\mathcal{L}_{\omega}^{*}\psi_{\sigma\omega}=\gamma^{it}_{\omega}
            \psi_{\omega},
            $$
            where $\gamma_{\omega}^{it}\in S^1,$   $\mathcal{L}_{\omega}^{*}$ denotes the adjoint of $\mathcal{L}_{\omega}$ and $\psi_{\omega}\in \mathcal{B}^*$, only has a measurable non-zero solution $\psi:=\{\psi_{\omega}\}_{\omega \in \Omega}$ when $t=0$. In this case $\gamma^0_{\omega}=1$ and $\psi_{\omega}(f)=\int f dm $ (up to a scalar multiplicative factor)  for $\mathbb{P}$-a.e. $\omega\in \Omega$.
\end{itemize}
In our present Anosov setting the bound (\ref{LD}) will be replaced with the following:
\begin{equation}\label{LD2}
\lVert \mathcal L_{\om}^{it, (n)}\rVert_{1,1} \le C\rho^n, \quad \text{for $t\in J$ and $n\ge 0$.}
\end{equation}
   Still in the present setting, we can not prove at the moment the equivalence between (L) and the aperiodicity condition although several of the technical steps which formed the skeleton of our proof of~\cite[Lemma 4.7.]{DFGTV} for expanding maps and functions of bounded variation can be transferred to Anosov maps and the anisotropic Banach spaces used in  this work.

\begin{proof}[Proof of Theorem \ref{thm:lclt}]

The proof assuming (L) follows now exactly as in the proof of Theorem C \cite{DFGTV}, with the following minor modifications.  We use Lemma \ref{dualtwist} to obtain the coding of the Birkhoff sums through powers of the twisted transfer operator.  The control of term (III) in the proof of Theorem C \cite{DFGTV} uses Lemma \ref{lem:UnifExpDecayY2} in place of Lemma 4.4 \cite{DFGTV}.
 The control of term (IV) in the proof of Theorem C \cite{DFGTV}  uses (L1) in place of the analogous condition (C5) in \cite{DFGTV}.
\end{proof}
 As mentioned in the introduction, Hafouta and Kifer \cite{HafoutaKifer}, Section 2.10,  formulated a sort of classical aperiodicity condition in the random setting which allowed them to recover condition (L). We now state their assumptions and verify that some of them hold in our setting. This and additional hypothesis on the driving map $\sigma$ will give us a new proof of the local central limit theorem. \\
 We first recall the Hafouta and Kifer assumptions adapted to our situations:
 \begin{itemize}
 \item HK A1: The probability measure $\mathbb{P}$ assigns positive measure to open sets, $\sigma$ is a homeomorphism and there exist $\omega_0\in \Omega$ and $m_0\in \mathbb{N}$ so that $\sigma^{m_0}\omega_0=\omega_0.$ Moreover  for each $i\in \{0, 1, \ldots, m_0-1\}$, there exists a neighborhood of $\sigma^i \omega_0$  on which the map $\omega \mapsto T_\omega$ is constant (note that on this neighborhood we also have that the map $\omega \mapsto \mathcal L_\omega$ is constant).
     \item HK A2: For each compact interval $J\subset \mathbb R$, the family of maps $\omega \rightarrow \mathcal{L}_{\omega}^{it},$ where $t \in J$, is equicontinuous at the points $\omega=\sigma^i\omega_0, 0\le i\le m_0$ with respect to the operator norm, and there exists a constant $B=B(J)\ge 1,$ such that $\mathbb{P}$-a.e.,
         \begin{equation}\label{PO}
         \lVert \mathcal L_{\om}^{it, (n)}\rVert_{1,1} \le B,
          \end{equation}
          for any $n\in \mathbb{N}$ and $t\in J.$
          \item HK A3: For any compact interval $J\subset \mathbb R$ that does not contain the origin, there exists constants $c=c(J)>0$ and $b=b(J)\in (0,1)$ such that
              \begin{equation}\label{QO}
\lVert {\bf L}_{it}^{s}\rVert_{1,1} \le cb^s,
              \end{equation}
             for any $s\in \mathbb{N}$ and $t\in J$, where the (deterministic) operator ${\bf L}_{it}$ is defined as ${\bf L}_{it}:=\mathcal{L}^{it, (m_0)}_{\omega_{0}}.$
\end{itemize}
Under these three assumptions, it was proved in Lemma 2.10.4 \cite{HafoutaKifer} that condition (L) holds.

We now adapt the previous requirements to our setting. Assumption HK A1 is easily satisfied by requiring that $\sigma$ is a homeomorphism that  has at least one periodic  point $\omega_0$ and by building the cocycle in a way that $\omega \mapsto T_\omega$ is locally constant at all points that belong to the orbit of $\omega_0$. Of course, we also need to  work with $\mathbb P$ that assigns positive measure to all open
nonempty subsets of $\Omega$.

 Assumption HK A3 is equivalent to requiring the classical aperiodicity condition for deterministic systems (see~\cite{hennion}).  Namely, it is sufficient to require that the spectral radius of $\mathcal{L}^{it, (m_0)}_{\omega_{0}}$ is strictly less than $1$ for all $t\neq 0$.  Then, taking $J\subset \mathbb R$ a  compact interval, we can find $c=c(J)>0$ and $b=b(J)\in (0,1)$ such that~\eqref{QO} holds  by  arguing as in~\cite[Proof of Lemma 4.7]{DFGTV} (although the argument is  simpler in our present setting   since we deal with a  deterministic situation).

 Assumption HK A2 should instead be checked under suitable conditions since it relies on the properties of the maps, of the observable and of the functional spaces: we now show that it holds for our Anosov maps. This plus the  necessary conditions in HK A1 and HK A3, will give us a proof of the LCLT in the non-arithmetic case.
To verify HK A2 we  first prove the equicontinuity property (under suitable conditions) in Lemma~\ref{EQUI} and then in Lemma~\ref{lem:strongLY} we will develop  a Lasota-Yorke  inequality from which we can obtain   the bound (\ref{PO}). The latter will be obtained from a Lasota--Yorke inequality for the twisted operator $\mathcal L_{\om}^{it, (n)}$ and with respect to the norm $\lVert\cdot\rVert_{1,1}$ (strong) and $\lVert\cdot\rVert_{0,2}$ (weak),  and the fact that $\lVert h\rVert_{0,2}\le \lVert h\rVert_{1,1}$ for any $h\in \mathcal{B}^{1,1}$,
as we did in (\ref{wsly_cont}) for the non-twisted  operator.
\begin{lemma}\label{EQUI}
Let us suppose $(\Omega, \mathcal F, \mathbb P, \sigma)$ is  an invertible and  ergodic measure-preserving dynamical  system verifying HK A1 and moreover  for each $i\in \{0, 1, \ldots, m_0-1\}$, the observable $g$ satisfies
\begin{equation}\label{x}
\lim_{\omega \to \sigma^i \omega_0} \lVert g(\omega, \cdot)-g(\sigma^i \omega_0, \cdot)\rVert_{C^2}=0.
\end{equation}
Furthermore, let $J\subset \R$ be a compact interval.
 Then, the family of maps $\{ \omega \mapsto \mathcal L_\omega^{it}: t\in J\}$ is equicontinuous in all points $\omega$ that belong to the orbit of $\omega_0$.
\end{lemma}

\begin{rmk}
Observe that it is not enough to simply prescribe that~\eqref{x} holds since we also need to make sure that this requirement  is not spoiled when we center our observable (see~\eqref{zeromean_cont}). It turns out that under suitable conditions we can make sure that~\eqref{x} is preserved  under centering.
Let us first recall   that~\eqref{DEC_cont} and~\eqref{wsly_cont} hold   for every $\omega \in \Omega$.
We assume that $\Omega$ is a metric space,  $\sigma \colon \Omega \to \Omega$ is a homeomorphism and that
\[
\text{ $\omega \mapsto T_\omega$ is locally constant at points that belong to the orbit of $\omega_0$.}
\]
Let us now modify slightly Proposition~\ref{prop:uniqueAcim} to ensure that under these additional assumptions we can say more about the top Oseledets space of our cocycle.

Set
\[
\mathcal Y=\{v\colon \Omega \to \mathcal B^{1,1}: \text{$v$ measurable and} \  \lVert v\rVert_\infty:=\sup_{\omega \in \Omega}\lVert v(\omega)\rVert_{1,1}<\infty \}.
\]
Then, $(\mathcal Y, \lVert \cdot \rVert_\infty)$ is a Banach space. Let $Y$ be a set of all $v\in \mathcal Y$ that  are continuous at  points $\sigma^i \omega_0$, $i=0,1, \ldots, m_0-1$.
We claim that $Y$ is a closed subset of $\mathcal Y$. Indeed, take a sequence $(v_n)_n \subset Y$ such that $v_n \to v$ in $\mathcal Y$ and fix  $i\in \{0, 1, \ldots, m_0-1\}$. Then, we have that
\[
\begin{split}
\lVert v(\omega)-v(\sigma^i \omega_0)\rVert_{1,1} &\le \lVert v(\omega)-v_n(\omega)\rVert_{1,1}+\lVert v_n(\omega)-v_n(\sigma^i \omega_0)\rVert_{1,1} \\
&\phantom{\le}+\lVert v_n(\sigma^i \omega_0)-v(\sigma^i \omega_0)\rVert_{1,1}\\
&\le 2\lVert v-v_n \rVert_\infty +\lVert v_n(\omega)-v_n(\sigma^i \omega_0)\rVert_{1,1}.
\end{split}
\]
Take $\varepsilon >0$ and choose $n$ such that
\[
\lVert v-v_n \rVert_\infty< \frac{\varepsilon}{3}.
\]
Since $v_n \in Y$, we have that
\[
\lVert v_n(\omega)-v_n(\sigma^i \omega_0)\rVert_{1,1}<\frac{\varepsilon}{3},
\]
whenever $\omega$ is sufficiently close to $\sigma^i \omega_0$. Hence,
\[
\lVert v(\omega)-v(\sigma^i \omega_0)\rVert_{1,1}<\varepsilon,
\]
whenever $\omega$ is sufficiently close to $\sigma^i \omega_0$.
 Therefore, $v\in Y$ and $Y$ is closed.

Set
\[
Z:=\{v\in Y; v(\omega) \ge 0 \ \text{and} \  v(\omega)(1)=1  \ \text{for $\omega \in \Omega$} \}.
\]
Then, $Z$ is a closed subset of $Y$ (see the argument in the proof of~Proposition~\ref{prop:uniqueAcim}) and hence it is a Banach space. We consider $\mathbb L \colon Z\to Z$ defined by
\[
(\mathbb L v)(\omega)=\mathcal L_{\sigma^{-1}\omega}v(\sigma^{-1}\omega), \quad \omega \in \Omega, \ v\in Z.
\]
In order to show that $\mathbb L$ is well-defined, we only need to note  that
\[
\omega \mapsto \mathcal L_{\sigma^{-1}\omega}v(\sigma^{-1}\omega)
\]
is continuous at $\sigma^i \omega_0$, $i\in \{0, \ldots, m_0-1\}$.
 However, this follows from the fact that $v\in Z$ (and thus $v\in Y$) and our assumption that $\omega \mapsto T_\omega$ (and thus also $\omega \mapsto \mathcal L_\omega$) is locally constant along the orbit of $\omega_0$. It follows from~\eqref{DEC_cont} that $\mathbb L$ has the unique fixed point $h^0\in Z$.
This easily implies that
\[
\omega \mapsto h_\omega^0(g(\omega, \cdot)), \quad h_\omega^0:=h^0(\omega)
\]
is continuous at $\sigma^i \omega_0$, $i=0, \ldots, m_0-1$ and therefore~\eqref{x} will remain valid even after centering.

\end{rmk}
\begin{proof}[Proof of Lemma~\ref{EQUI}]
We will prove the desired equicontinuity property in $\omega_0$. The argument for all other points in the  orbit of $\omega_0$ is completely analogous. Observe that for all $\omega \in \Omega$ sufficiently close to $\omega_0$, we have that $\mathcal L_\omega=\mathcal L_{\omega_0}$. Therefore, for all $\omega$ close to $\omega_0$, we have that
\[
(\mathcal L_\omega^{it}-\mathcal L_{\omega_0}^{it})(h)=\mathcal L_{\omega_0}((e^{it g(\omega, \cdot)}-e^{it g(\omega_0, \cdot)} )h),
\]
and thus
\[
\lVert (\mathcal L_\omega^{it}-\mathcal L_{\omega_0}^{it})(h) \rVert_{1,1}\le \lVert \mathcal L_{\omega_0}\rVert_{1,1} \cdot  \lVert (e^{it g(\omega, \cdot)}-e^{it g(\omega_0, \cdot)} )h \rVert_{1,1},
\]
for each $h\in \mathcal B^{1,1}$. Observe that
\[
\begin{split}
&\lVert (e^{it g(\omega, \cdot)}-e^{it g(\omega_0, \cdot)} )h \rVert_{1,1} = \\
&=\max \{\lVert (e^{it g(\omega, \cdot)}-e^{it g(\omega_0, \cdot)} )h \rVert_{0,1}^\sim,  \lVert (e^{it g(\omega, \cdot)}-e^{it g(\omega_0, \cdot)} )h \rVert_{1,2}^\sim \}.
\end{split}
\]
As in the proofs of Propositions~\ref{twistcty} and \ref{twistQC}, we need to estimate \[ \lVert e^{it g(\omega, \cdot)}-e^{it g(\omega_0, \cdot)}\rVert_{C^2}.\]

Take $x\in X$. By applying the mean-value theorem for the map $z\mapsto e^{itz}$, we see that
\[
\lvert e^{itg(\omega, x)}-e^{itg(\omega_0, x)}\rvert \le \lvert t\rvert \cdot \lvert g(\omega, x)-g(\omega_0, x)\rvert.
\]
Thus,
\[
\lVert e^{it g(\omega, \cdot)}-e^{it g(\omega_0, \cdot)}\rVert_{C^0} \le \lvert t\rvert \cdot \lVert g(\omega, \cdot)-g(\omega_0,\cdot)\rVert_{C^0},
\]
which implies that
\begin{equation}\label{815}
\lVert e^{it g(\omega, \cdot)}-e^{it g(\omega_0, \cdot)}\rVert_{C^0} \le \max \{\lvert t\rvert: t\in J \}\lVert g(\omega, \cdot)-g(\omega_0,\cdot)\rVert_{C^2}.
\end{equation}
Moreover, we have that
\[
\begin{split}
\partial^j (e^{it g(\omega, \cdot)}-e^{it g(\omega_0, \cdot)}) &=it e^{it g(\omega, \cdot)} \partial^j (g(\omega, \cdot))-it e^{itg(\omega_0, \cdot)} \partial^j (g(\omega_0, \cdot)) \\
&=it e^{it g(\omega, \cdot)} \partial^j (g(\omega, \cdot))-it e^{it g(\omega_0, \cdot)} \partial^j (g(\omega, \cdot)) \\
&\phantom{=}+it e^{it g(\omega_0, \cdot)} \partial^j (g(\omega, \cdot))-it e^{itg(\omega_0, \cdot)} \partial^j (g(\omega_0, \cdot)),
\end{split}
\]
for each $j\in \{1, \ldots, d\}$. By~\eqref{815}, we have that
\[
\begin{split}
\lVert \partial^j (e^{it g(\omega, \cdot)}-e^{it g(\omega_0, \cdot)}) \rVert_{C^0} &\le  \max \{\lvert t\rvert^2: t\in J \}\lVert g(\omega, \cdot)-g(\omega_0,\cdot)\rVert_{C^2}\cdot \lVert g(\omega, \cdot)\rVert_{C^2} \\
&\phantom{\le}+\max \{\lvert t\rvert: t\in J \}\lVert g(\omega, \cdot)-g(\omega_0, \cdot)\rVert_{C^2},
\end{split}
\]
for every $j\in \{1, \ldots, d\}$. Thus,
\begin{equation}\label{tr}
\max_{1\le j\le d}\lVert \partial^j (e^{it g(\omega, \cdot)}-e^{it g(\omega_0, \cdot)}) \rVert_{C^0} \le C\lVert g(\omega, \cdot)-g(\omega_0, \cdot)\rVert_{C^2},
\end{equation}
for some $C>0$ which is independent on $t$ and $\omega$. Finally, for each $k, j\in \{1, \ldots, d\}$, we have that
\[
\begin{split}
\partial^k\partial^j (e^{it g(\omega, \cdot)}-e^{it g(\omega_0, \cdot)}) &=-t^2e^{itg(\omega, \cdot)}  \partial^k (g(\omega, \cdot)) \partial^j (g(\omega, \cdot))\\
&\phantom{=}+it e^{it g(\omega, \cdot)}\partial^k\partial^j (g(\omega, \cdot)) \\
&\phantom{=}+t^2e^{itg(\omega_0, \cdot)}  \partial^k (g(\omega_0, \cdot)) \partial^j (g(\omega_0, \cdot))\\
&\phantom{=}-it e^{it g(\omega_0, \cdot)}\partial^k\partial^j (g(\omega_0, \cdot)).
\end{split}
\]
Observe that
\[
\begin{split}
it e^{it g(\omega, \cdot)}\partial^k\partial^j (g(\omega, \cdot))-it e^{it g(\omega_0, \cdot)}\partial^k\partial^j (g(\omega_0, \cdot))&=it e^{it g(\omega, \cdot)}\partial^k\partial^j (g(\omega, \cdot))\\
&\phantom{=}-it e^{it g(\omega, \cdot)}\partial^k\partial^j (g(\omega_0, \cdot))\\
&\phantom{=}+it e^{it g(\omega, \cdot)}\partial^k\partial^j (g(\omega_0, \cdot))\\
&\phantom{=}-it e^{it g(\omega_0, \cdot)}\partial^k\partial^j (g(\omega_0, \cdot)).
\end{split}
\]
Thus (using~\eqref{815}),
\[
\begin{split}
& \lVert it e^{it g(\omega, \cdot)}\partial^k\partial^j (g(\omega, \cdot))-it e^{it g(\omega_0, \cdot)}\partial^k\partial^j (g(\omega_0, \cdot)) \rVert_{C^0} \\
&\le \max \{\lvert t\rvert: t\in J\}\lVert g(\omega, \cdot)-g(\omega_0, \cdot)\rVert_{C^2} \\
&\phantom{\le}+\max \{\lvert t\rvert^2: t\in J\}\lVert g(\omega, \cdot)-g(\omega_0, \cdot)\rVert_{C^2} \lVert g(\omega_0, \cdot)\rVert_{C^2}.
\end{split}
\]
On the other hand,
\[
\begin{split}
&-t^2e^{itg(\omega, \cdot)}  \partial^k (g(\omega, \cdot)) \partial^j (g(\omega, \cdot)) +t^2e^{itg(\omega_0, \cdot)}  \partial^k (g(\omega_0, \cdot)) \partial^j (g(\omega_0, \cdot)) \\
&=-t^2e^{itg(\omega, \cdot)}  \partial^k (g(\omega, \cdot)) \partial^j (g(\omega, \cdot))+t^2e^{itg(\omega_0, \cdot)}  \partial^k (g(\omega, \cdot)) \partial^j (g(\omega, \cdot)) \\
&\phantom{=}-t^2e^{itg(\omega_0, \cdot)}  \partial^k (g(\omega, \cdot)) \partial^j (g(\omega, \cdot))+t^2e^{itg(\omega_0, \cdot)}  \partial^k (g(\omega_0, \cdot)) \partial^j (g(\omega, \cdot)) \\
&\phantom{=}-t^2e^{itg(\omega_0, \cdot)}  \partial^k (g(\omega_0, \cdot)) \partial^j (g(\omega, \cdot))+t^2e^{itg(\omega_0, \cdot)}  \partial^k (g(\omega_0, \cdot)) \partial^j (g(\omega_0, \cdot))
\end{split}
\]
Hence, \eqref{815} implies that
\[
\begin{split}
&\lVert -t^2e^{itg(\omega, \cdot)}  \partial^k (g(\omega, \cdot)) \partial^j (g(\omega, \cdot)) +t^2e^{itg(\omega_0, \cdot)}  \partial^k (g(\omega_0, \cdot)) \partial^j (g(\omega_0, \cdot)) \rVert_{C^0} \\
&\le \max \{\lvert t\rvert^3: t\in J \}\lVert g(\omega, \cdot)-g(\omega_0,\cdot)\rVert_{C^2} \cdot \lVert g(\omega, \cdot)\rVert_{C^2}^2\\
&\phantom{\le}+t^2\lVert g(\omega, \cdot)-g(\omega_0,\cdot)\rVert_{C^2} \cdot \lVert g(\omega, \cdot)\rVert_{C^2}\\
&\phantom{\le}+t^2\lVert g(\omega, \cdot)-g(\omega_0,\cdot)\rVert_{C^2} \cdot \lVert g(\omega_0, \cdot)\rVert_{C^2}.
\end{split}
\]
We conclude that (by increasing $C$)  we have that
\[
\begin{split}
&\sup_{1\le k, j\le d}  \lVert -t^2e^{itg(\omega, \cdot)}  \partial^k (g(\omega, \cdot)) \partial^j (g(\omega, \cdot)) +t^2e^{itg(\omega_0, \cdot)}  \partial^k (g(\omega_0, \cdot)) \partial^j (g(\omega_0, \cdot))\rVert_{C^0} \\
&\le C\lVert g(\omega, \cdot)-g(\omega_0,\cdot)\rVert_{C^2}.
\end{split}
\]
Thus,
\[
 \lVert e^{it g(\omega, \cdot)}-e^{it g(\omega_0, \cdot)}\rVert_{C^2} \le C\lVert g(\omega, \cdot)-g(\omega_0,\cdot)\rVert_{C^2},
\]
and
\[
\lVert \mathcal L_\omega^{it}-\mathcal L_{\omega_0}^{it}\rVert \le C\lVert g(\omega, \cdot)-g(\omega_0,\cdot)\rVert_{C^2},
\]
for $t\in J$ and $\omega$ in a neighborhood of $\omega_0$.
The conclusion of the lemma follows directly from~\eqref{x}.
\end{proof}

As we already announced, we now prove a Lasota--Yorke inequality for the twisted operator.
\begin{lemma}\label{lem:strongLY}
For each $t\in \R$, there exist $A_t, B_t> 0$, $0<\gamma_t<1$ such that for every $n\geq 0, h \in \mc B^{1,1}$ and \paeom,
\[
\| \mcl_\om^{it, (n)} h\|_{1,1} \leq A_t \ga_t^n \| h\|_{1,1} + B_{t} \|h\|_{0,2}.
\]
Moreover, for each $J\subset \mathbb R$ compact interval, we have that
\[
\sup_{t\in J}\max \{A_t, B_t \} <\infty.
\]
\end{lemma}
\begin{proof}
(Sketch). The proof follows verbatim the proof of Lemma 6.3 in \cite{GL}, with two  differences. First, we work with composition of maps, but if they are close enough we can easily adapt the {\em deterministic} arguments (we recall that this was explicitly emphasized in~\cite[Section 7]{GL}, in particular allowing for a \textit{random version} of \cite[Lemma 3.3]{GL} to be applied).  Second, since we use the twisted operator instead of the usual one, in the various estimates in the proof of Lemma 6.3 \cite{GL} we find the extra multiplicative factor $e^{it S_n g(\omega, \cdot)}$.
The proof of Lemma 6.3 \cite{GL} is done by induction on the index $p$ and the first step is to get a weak version of the Lasota--Yorke, namely for each $t\in \R$, there exists $C_t\geq 0$ such that for every $n\geq 0$ and \paeom,
\begin{equation}\label{eq:weakLYtwist}
\| \mcl_\om^{it, (n)} \|_{0,q} \leq C_t.
\end{equation}
We now prove \eqref{eq:weakLYtwist} (with $q=1$) to show how to handle the additional multiplicative factor. We use the notation as in~\cite[p.202]{GL}.
Recall that
\[
\mathcal L_\omega^{(n)} (h)(x)=\frac{h((T_\omega^{(n)})^{-1}(x))}{\lvert \det DT_\omega^{(n)} ((T_\omega^{(n)})^{-1}(x))\rvert},
\]
and therefore
\[
\mathcal L_\omega^{it, (n)}(h)(x)=\frac{e^{it S_n g(\omega, (T_\omega^{(n)})^{-1}(x))}h((T_\omega^{(n)})^{-1}(x))}{\lvert \det DT_\omega^{(n)} ((T_\omega^{(n)})^{-1}(x))\rvert},
\]
for each $h\in C^r$ and $x\in X$. Thus,
\[
\int_W \mathcal L_\omega^{it, (n)}h \cdot \varphi=\int_{(T_\omega^{(n)})^{-1}(W)}\bar{h}_n\cdot \varphi \circ T_\omega^{(n)} \cdot J_W T_\omega^{(n)},
\]
where $J_W T_\omega^{(n)}$ is the Jacobian of $T_\omega^{(n)} \colon (T_\omega^{(n)})^{-1}(W) \to W$ and
\[
\bar{h}_n:=\frac{he^{itS_ng(\omega, \cdot)}}{\lvert \det DT_\omega^{(n)} \rvert}.
\]
Let $\varphi_j = \varphi \circ T_\omega^{(n)} \cdot \rho_j$, where
$\rho_1,\ldots,\rho_\ell$ is a partition of unity on $(T_\omega^{(n)})^{-1}W$, as provided by (the random analogue of) Lemma~3.3~\cite{GL} (using $\gamma=1$), and $W_1,\ldots,W_\ell$ the corresponding admissible leaves such that $(T_{\omega}^{(n)})^{-1}(W)\subset \cup_{j=1}^\ell W_j$.
Hence, \cite[(6.2)]{GL}  becomes
\[
\bigg{\lvert} \int_{W_j}\bar{h}_n\cdot \varphi_j \cdot J_WT_\omega^{(n)} \bigg{\rvert }\le C\lVert h\rVert_{0,1} \bigg{\lvert}\lvert \det DT_\omega^{(n)} \rvert^{-1} \cdot e^{it S_ng(\omega, \cdot)}\cdot \varphi_j \cdot J_WT_\omega^{(n)}  \bigg{\rvert}_{C^1(W_j)}.
\]
Note that
\[
\bigg{\lvert}\lvert \det DT_\omega^{(n)} \rvert^{-1} \cdot e^{it S_ng(\omega, \cdot)}\cdot \varphi_j \cdot J_WT_\omega^{(n)}  \bigg{\rvert}_{C^1(W_j)} \le \bigg{\lvert}\lvert \det DT_\omega^{(n)} \rvert^{-1} \cdot  J_WT_\omega^{(n)}  \bigg{\rvert}_{C^1(W_j)} \cdot  \lvert \varphi_j\rvert_{C^1(W_j)}\cdot
\lvert e^{itS_ng(\omega, \cdot)}\rvert_{C^1(W_j)}.
\]
It follows from~\cite[Lemma 6.2]{GL} that
\[
 \sum_{j\le \ell} \lvert \lvert \det DT_\om^{(n)}\rvert^{-1}\cdot J_WT_\om^{(n)}
 \rvert_{C^1(W_j)}\le C.
\]
In addition, from the argument at the bottom of~\cite[p. 203]{GL}, it follows that
\[
 \lvert \varphi_j\rvert_{C^1(W_j)}\le \lvert \varphi \circ T_\om^{(n)} \rvert_{C^1(W_j)} \cdot \lvert \rho_j\rvert_{C^1(W_j)}\le C.
\]
Hence, in order to complete the proof of the weak Lasota--Yorke inequality, it is sufficient to show that
\begin{equation}\label{sad}
\lvert e^{itS_n g(\om, \cdot)}\rvert_{C^1(W_j)}\le C.
\end{equation}
Note that
\begin{equation*}
\begin{split}
 &\lvert e^{itS_n g(\om, \cdot)}\rvert_{C^0(W_j)}=1 \quad \text{and} \\
  &\lvert \partial^\al(e^{itS_n g(\om, \cdot)})\rvert_{C^0(W_j)}=\lvert t\rvert \cdot \lvert \partial^\al(S_n g(\om, \cdot))\rvert_{C^0(W_j)}
\le \lvert t\rvert \sum_{i=0}^{n-1}\lvert \partial^\al(g(\sig^i\om, T_\om^{(i)}(\cdot)))\rvert_{C^0(W_j)}.
 \end{split}
 \end{equation*}
In order to  bound $\lvert \partial^\al(g(\sig^i\om, T_\om^{(i)}(\cdot)))\rvert_{C^0(W_j)}$, we proceed as in~\cite[(4.3)]{DL}. For each $i$ and $x, y\in W_j$, we have
\begin{equation*}
\begin{split}
 \frac{\lvert g(\sig^i\om, T_\om^{(i)}x)-g(\sig^i\om, T_\om^{(i)} y)\rvert}{d(x,y)}&=\frac{\lvert g(\sig^i\om, T_\om^{(i)}x)-g(\sig^i\om, T_\om^{(i)}y)\rvert}{d(T_\om^{(i)} x,T_\om^{(i)} y)}\cdot \frac{d(T_\om^{(i)} x,T_\om^{(i)} y)}{d(x,y)}\\
 &\le
 C\nu^i \esssup_{\om \in \Om} \lVert g(\om, \cdot)\rVert_{C^1},
\end{split}
\end{equation*}
and thus
\[
\lvert \partial^\al(g(\sig^i\om, T_\om^{(i)}(\cdot)))\rvert_{C^0(W_j)} \le C\nu^i\esssup_{\om \in \Om} \lVert g(\om, \cdot)\rVert_{C^1}.
\]
In view of~\eqref{obs_cont}, \eqref{sad} holds.
Now one can repeat arguments in~\cite{GL} to obtain the weak Lasota--Yorke inequality for the twisted cocycle, \eqref{eq:weakLYtwist}. The proof of the strong Lasota--Yorke inequality can be obtained in a similar manner.
\end{proof}
Using Lemmas~\ref{EQUI} and~\ref{lem:strongLY} we get:
\begin{thm}\label{HKLCLT}
Suppose that conditions HK A1 and HK A3 hold and the observable $g$ satisfies~\eqref{x}. Then the same conclusions as those of Theorem C hold.
\end{thm}

\section{Piecewise hyperbolic dynamics}
\label{sect:pwh}

In this section, we apply the previous theory to obtain statistical laws for the random compositions $T^{(n)}_\omega=T_{\sigma^{n-1}\omega}\circ\cdots\circ T_{\sigma\omega}\circ T_\omega$ of piecewise uniformly hyperbolic maps $T_\omega$ of the type studied in \cite{DL}.
The class of maps $T_\omega$ considered contains piecewise toral automorphisms and piecewise hyperbolic maps with bounded derivatives;  see Remark 2.2 \cite{DL}.
\subsection{Preliminaries}
We follow the construction of \cite{DL}.
Let $X$ be a two-dimensional compact Riemannian manifold, possibly with boundary and not necessarily connected and let $T\colon X \to X$ be a piecewise hyperbolic map in the sense
of~\cite{DL}.
That is, the domain $X$ is broken into a finite number of pairwise disjoint open regions $\{X^+_i\}$ with piecewise $C^1$ boundary curves of finite length, such that $\bigcup_i \overline{X^+_i}=X$.
The image of each $X^+_i$ under $T$ is denoted $X^-_i=T(X^+_i)$; we assume that $\bigcup \overline{X^-_i}=X$.
The sets $\mathcal{S}^\pm:=X\setminus\bigcup_i X^\pm_i$ are the ``singularity sets'' for $T$ and $T^{-1}$, respectively.
Assume that $T$ is a $C^2$ diffeomorphism from the complement of $\mathcal{S}^+$ to the complement of $\mathcal{S}^-$, and that for each $i$, there is a $C^2$ extension of $T$ to $\overline{X^+_i}$.
On each $X_i$, the map $T$ is uniformly hyperbolic:  there are two continuous, strictly $DT$-invariant families of cones $C^s$ and $C^u$ defined on $X\setminus (\mathcal{S}^+\cup \partial X)$ satisfying
\begin{eqnarray*}
\lambda&:=&\inf_{x\in X\setminus \mathcal{S}^+}\inf_{v\in C^u} \frac{\|DTv\|}{\|v\|}>1,\\
\mu&:=&\inf_{x\in X\setminus \mathcal{S}^+}\inf_{v\in C^s} \frac{\|DTv\|}{\|v\|}<1,\\
\mu_+^{-1}&:=&\inf_{x\in X\setminus \mathcal{S}^-}\inf_{v\in C^s} \frac{\|DT^{-1}v\|}{\|v\|}>1.
\end{eqnarray*}
Assume that vectors tangent to the singularity curves in $\mathcal{S^-}$ are bounded away from $C^s$.
The singularity curves and their images and preimages should not intersect at too many points.
Denote by $\mathcal{S}_n^-$ (resp.\ $\mathcal{S}^+_n$) the set of singularity curves for $T^{-n}$ (resp.\ $T^n$), and let $M(n)$ denote the maximum number of singularity curves that meet at a single point.
Assume that there is an $\alpha_0$ and an integer $n_0>0$ such that $\lambda\mu^{\alpha_0}>1$ and $(\lambda\mu^{\alpha_0})^{n_0}>M(n_0)$;  this condition is satisfied if $M(n)$ has polynomial growth, for example.

For each $n\in\mathbb{N}$, let $\mathcal{K}_n$ be the set of connected components of $X\setminus \mathcal{S}_n^+$, and let $C^1(\overline{K},\mathbb{R})$ be the set of functions $\varphi\in C^1(\mathring{K},\mathbb{R})$ with $C^1$ extension in a neighbourhood of $\overline{K}$.
Let $(C^1_{\mathcal{S}_n^+})':=\{\varphi\in L^\infty(X): \varphi\in C^1(\overline{K},\mathbb{R})\ \forall K\in\mathcal{K}_n\}.$
If $h\in (C^1_{\mathcal{S}_n^+})'$ is an element of the dual of $C^1_{\mathcal{S}_n^+}$, then $\mathcal{L}:(C^1_{\mathcal{S}_n^+})'\to (C^1_{\mathcal{S}_{n-1}^+})'$ acts on $h$ by
$$\mathcal{L}h(\varphi)=h(\varphi\circ T)\quad\forall \varphi\in C^1_{\mathcal{S}_{n-1}^+}.$$

In order to obtain useful spectral information from $\mathcal{L}$, its action is restricted to a Banach space $\mathcal{B}$, analogous to the space $\mathcal{B}^{p,q}=\mathcal{B}^{1,1}$ in Section \ref{sec:prelim}.
We now briefly outline the construction of the norms on $\mathcal{B}$ and an associated ``weak'' space $\mathcal{B}_w$;  see \cite{DL} for details.
The norms are defined using ``admissible leaves'' $W$ in a set of admissible leaves $\Sigma$.
These leaves are smooth curves in approximately the stable direction, and are analogues of the $\psi_i\circ ({\rm Id},\chi)$ defined in Section \ref{sec:prelim}. Since we are going to recall several times estimates in \cite{DL}, we intend to comply with the notation there. In particular the functions $\chi$ defined on the charts will now become $F$ and the image of the graph of $F$, namely the {\em admissible leaves}, will be denoted with $G_F.$
For $\alpha,\beta,q<1$ such that $0<\beta\le\alpha\le 1-q\le\alpha_0$ let $C^\alpha(W, \mathbb C)$ denote the set of continuous complex-valued functions on $W$ with H\"older exponent $\alpha$ and define the norm
\begin{equation}\label{TF}
 \lvert \varphi \rvert_{W, \alpha, q}:=\lvert W\rvert^\alpha \cdot \lvert \varphi \rvert_{C^q(W, \mathbb C)},
\end{equation}
where $|W|$ denotes unnormalised induced Riemannian volume of $W$.
For $h\in C^1(X, \mathbb C)$ we define the weak norm of $h$ by
\[
 \lvert h\rvert_w=\sup_{W\in \Sigma}\sup_{\substack{\varphi \in C^1(W, \mathbb C) \\ \lvert \varphi \rvert_{C^1(W, \mathbb C)}\le 1}}\bigg{\lvert} \int_W h\varphi \, dm \bigg{\rvert}
\]
and the strong norm by \[\lVert h\rVert=\lVert h\rVert_s+b\lVert h\rVert_u,\] where
the strong stable norm is
\begin{equation}
\label{ssnorm}
 \lVert h\rVert_s=\sup_{W\in \Sigma}\sup_{\substack{\varphi \in C^1(W, \mathbb C) \\ \lvert \varphi \rvert_{W, \alpha, q}\le 1}}\bigg{\lvert} \int_W h\varphi \, dm \bigg{\rvert}
\end{equation}
and the strong unstable norm is
\begin{equation}
\label{sunorm}
 \lVert h\rVert_u=\sup_{\epsilon \le \epsilon_0}\sup_{\substack{W_1, W_2\in \Sigma \\ d_{\Sigma} (W_1, W_2)\le \epsilon}}\sup_{\substack {\lvert \varphi_i\rvert_{C^1(W_i, \mathbb C)}
 \le 1 \\d_q(\varphi_1, \varphi_2)\le \epsilon}}\frac{1}{\epsilon^\beta}\bigg{\lvert}
 \int_{W_1}h\varphi_1\, dm-\int_{W_2}h\varphi_2\, dm\bigg{\rvert},
\end{equation}
where $d_\Sigma$ and $d_q$ are defined precisely in \S3.1 \cite{DL}.
In comparison to the setting in Section \ref{sec:prelim}, the norm $|\cdot|_w$ plays the role of $\|\cdot\|_{p-1,q+1}=\|\cdot\|_{0,2}$, and the norm $\|\cdot\|$ plays the role of $\|\cdot\|_{p,q}=\|\cdot\|_{1,1}$.

Let $\mathcal B$ be the completion of $C^1(X, \mathbb C)$ with respect to the norm $\lVert \cdot \rVert$. Similarly, we define
$\mathcal B_w$ to the completion of $C^1(X, \mathbb C)$ with respect to the norm $\lvert \cdot \rvert_w$.

We recall that the elements of $\mathcal B$ are distributions. More precisely, there exists $C>0$ such that   any $h\in \mathcal B$ induces a
linear functional
$\varphi \to h(\varphi)$ with the property that
\begin{equation}\label{dist}
 \lvert h(\varphi)\rvert \le C\lvert h\rvert_w \lvert \varphi\rvert_{C^1}, \quad \text{for $\varphi \in C^1(X, \mathbb C)$,}
\end{equation}
see~\cite[Remark 3.4]{DL} for details.
In particular, for $h\in C^1(X, \mathbb C)$ we have that (see~\cite[Remark 2.5]{DL})
\begin{equation}\label{d} h(\varphi)=\int_X h\varphi,  \quad \text{for $\varphi \in C^1(X, \mathbb C)$.} \end{equation}
We say that $h\in \mathcal B$ is nonnegative and write $h\ge 0$ if $h(\varphi)\ge 0$ for any $\varphi \in C^1(X,\mathbb{R})$ such that $\varphi \ge 0$.
Finally, we recall (see~\cite[Section 2.1]{DL}) that for $h\in L^1(X,\mathbb{C})$,
\begin{equation}\label{to}
 \mcl h=\bigg{(}\frac{h} {\lvert \det DT\rvert}\bigg{)}\circ T^{-1}.
\end{equation}
\begin{proposition}\label{aux}
 We have that
 \[
  (\mcl h)(\varphi)=h(\varphi \circ T), \quad \text{for $h\in \mathcal B$ and $\varphi \in C^1(X, \mathbb C)$.}
 \]

\end{proposition}

\begin{proof}
 For $h\in C^1(X, \mathbb C)$ the desired conclusion can be easily obtained from~\eqref{d} and~\eqref{to} by using a change of variables. This immediately implies that the conclusion holds
 for any $h\in \mathcal B$.
\end{proof}

\subsection{Building the cocycle}

This section follows the material in Section \ref{sec:InitialCocycle},
replacing $(\mathcal{B}^{1,1},\|\cdot\|_{1,1})$ with $(\mathcal{B},\|\cdot\|)$ and $(\mathcal{B}^{0,2},\|\cdot\|_{0,2})$ with $(\mathcal{B}_w,\|\cdot\|_w)$.
We have included this material to make the relevant references to \cite{DL} transparent.

\cite[Theorem 2.8]{DL} implies that the associated transfer operator
$\mcl_T$ is quasicompact on $\mathcal B$, $1$ is a simple
eigenvalue and there are no other eigenvalues of modulus $1$. This in particular implies (using the terminology as in~\cite[Definition 2.6]{CR}) that $\mcl_T$ is exact in
$\{h\in \mathcal B: h(1)=0\}$.

Let $\Gamma_{B_*}$ and $X_\epsilon$  be the sets of maps as defined in~\cite[Section 2.4]{DL}. By applying~\cite[Lemma 6.1]{DL}, we find that there exists $C>0$ such that
\[
 \sup_{\lVert h\rVert \le 1}\lvert (\mcl_{T'}-\mcl_T)h\rvert_w  \le C\epsilon^\beta \quad \text{for $T' \in X_\epsilon$.}
\]
It then follows from~\cite[Lemma 3.5]{DL} and the discussion on~\cite[Section 2.4]{DL} that there exist $\epsilon, A>0$ and $c\in (0, 1)$ such that
for any $T'\in \Gamma_\epsilon$, we have that
\begin{itemize}
 \item the unit ball in $\mathcal B$ is relatively compact in $\mathcal B_w$;
 \item  $\lvert \mcl_{T'}^n h\rvert_w\le A\lvert h\rvert_w$ for each $n\in \N$ and $h\in \mathcal B$;
 \item
 $
  \lVert \mcl_{T'}^n h\rVert\le Ac^n \lVert h\rVert +A\lvert h\rvert_w$ for each $n\in \N$ and $h\in \mathcal B$.

\end{itemize}
Consider now a family of operators
\[
\mathcal P=\{\mcl_{T'}: \ T'\in X_\epsilon \}.
\]
It then follows from~\cite[Proposition 2.10]{CR} (applied to the case where $\lVert \cdot \rVert=\lvert \cdot \rvert_w$ and $\lvert \cdot \rvert_v=\lVert \cdot \rVert$)
that there exists $0<\epsilon'\le \epsilon$, $D, \lambda >0$ such that for any $T_1, \ldots, T_n \in X_{\epsilon'}$, we have that
\begin{equation}\label{dec}
 \lVert \mcl_{T_n}\cdots \mcl_{T_2}\mcl_{T_1}h\rVert \le De^{-\lambda n}\lVert h\rVert \quad \text{for $h\in \mathcal B$ satisfying $h(1)=0$,}
\end{equation}
where $\mcl_{T_i}$ denotes the transfer operator associated with $T_i$.
From now on, we replace $\epsilon'$ with $\epsilon$ so that (\ref{dec}) holds for $T_1,\ldots,T_n\in X_\epsilon$.

We now build our cocycle by prescribing that for each $\om \in \Om$, $T_\om \in  X_{\epsilon'}$ and we consider
$\mcl_\om$  which is the transfer operator associated to $T_\omega$.
Then, it follows readily from~\eqref{dec} that
\begin{equation}\label{DEC}
 \lVert \mcl_\om^{(n)} h\rVert \le De^{-\lambda n} \lVert h\rVert \quad \text{for any $\om \in \Om$, $n\in \mathbb N$ and $h\in \mathcal B$, $h(1)=0$.}
\end{equation}

In addition, by decreasing $\epsilon$ if necessary, we have (see the proof of~\cite[Lemma 6.3]{DL})  that there exist $a\in (0, 1)$ and $B>0$ such that
\begin{equation}\label{wsly}
 \lvert \mcl_\om^{(n)} h\rvert_w \le B\lvert h\rvert_w \quad \text{and} \quad \lVert \mcl_\om^{(n)} h\rVert \le Ba^n\lVert h\rVert+B\lvert h\rvert_w,
\end{equation}
for every $\om \in \Om$, $n\in \mathbb N$ and $h\in \mathcal B$.  In particular, there exists $K>0$ such that
\begin{equation}\label{ub}
 \lVert \mcl_\om h\rVert \le K\lVert h\rVert \quad \text{for $\om \in \Om$ and $h\in \mathcal B$.}
\end{equation}

\subsubsection{$\mathbb{P}$-continuity of $\omega\mapsto \mathcal{L}_\omega$}
We assume $\Omega$ is a Borel subset of a complete separable metric space, $\mathcal{F}$ is the Borel sigma-algebra and $\sigma$ is a homeomorphism.
Unfortunately, in this (piecewise-hyperbolic) setting we are unable to establish strong measurability of the map $\omega \mapsto \mathcal L_\omega$ under the assumption that $\omega \mapsto T_\omega$ is measurable.  In order to be able to apply the weaker version of MET from~\cite{FLQ2}, we ask instead that $\omega \mapsto T_\omega$ is measurable and that it has a countable range.

\subsubsection{Quasi-compactness of the cocycle $\mcl$ and existence of Oseledets splitting}

Similarly to the description at the end of Section \ref{sect:qc0}, by Lemma 2.1 \cite{DFGTV}, the inequalities (\ref{wsly}) and (\ref{ub}) imply that the cocycle $\mathcal{R}$ is quasi-compact.
By quasi-compactness and $\mathbb{P}$-continuity of $\mathcal{L}$, the multiplicative ergodic theorem (Theorem 17, \cite{FLQ2}) yields the existence of a unique $\mathbb{P}$-continuous Oseledets splitting
$$\mathcal{B}_{1,1}=\left(\bigoplus_{j=1}^l Y_j(\omega)\right)\oplus V(\omega),$$
where each component of the splitting is equivariant under $\mathcal{L}_\omega$.
The $Y_j(\omega)$ are finite-dimensional and have corresponding (finite or infinite) sequence of  Lyapunov exponents $0=\lambda_1>\lambda_2>\ldots$

\subsubsection{One-dimensionality of the top Oseledets space}

The material in section \ref{sec:topspace} of the present work can be reused verbatim in the piecewise hyperbolic setting, replacing $(\mathcal{B}^{1,1},\|\cdot\|_{1,1})$ with $(\mathcal{B},\|\cdot\|)$ and $(\mathcal{B}^{0,2},\|\cdot\|_{0,2})$ with $(\mathcal{B}_w,\|\cdot\|_w)$. In particular we can construct a unique family of probability measures $(h^0_{\omega})_{\omega\in \Omega}\subset \mathcal{B}$ such that for $\mathbb{P}$-a.e. $\omega \in \Omega,$ $\mathcal{L}_{\omega}h^0_{\omega}=h^0_{\sigma\omega}.$

\subsection{The twisted cocycle}

Our observable will be a map $g\colon \Omega \times X \to \mathbb R$ such that $g(\om, \cdot)\in C^1$ for $\om \in \Om$ and
\begin{equation}\label{obs}
 M:=\esssup_{\om \in \Om} \lVert g(\om, \cdot)\rVert_{C^1} <\infty.
\end{equation}
We assume that $g$ is $\omega$-fibrewise centred: for \paeom, $h_\om^0 (g(\om, \cdot))=0.$

For $g\in C^1(X, \mathbb C)$ and $h\in \mathcal B$, we can  introduce $g\cdot h \in \mathcal B$ as in Section~\ref{sec:prelim}. Furthermore, for $\om \in \Om$, $\theta \in \mathbb C$, and $h\in \mathcal B$ set
$\mcl_\om^{\theta} (h)=\mcl_\om(e^{\theta g(\om, \cdot)}h).$
We will need the following lemma, which is analogous to Lemma 3.2~\cite{GL}.
\begin{lemma}
\label{EST}
 For $h\in \mathcal B$ and $g\in C^1(X, \mathbb C)$, we have that
 \[
  \lVert gh\rVert \le C\lvert g\rvert_{C^1}\lVert h\rVert,
 \]
 for some $C>0$, independent of $g$ and $h$.

\end{lemma}

\begin{proof}
It is sufficient to establish the desired conclusion for $h\in C^1(X, \mathbb C)$.
Note that
\begin{equation}\label{qx}
 \lVert gh\rVert=\lVert gh\rVert_s+b\lVert gh\rVert_u.
\end{equation}
We have
\begin{equation}\label{est1}
 \lVert gh\rVert_s=\sup_{W\in \Sigma}\sup_{\substack{\varphi \in C^1(W, \mathbb C) \\ \lvert \varphi \rvert_{W, \alpha, q}\le 1}}\bigg{\lvert} \int_W h\varphi g\, dm \bigg{\rvert}
\le \lvert g\rvert_{C^1} \cdot \lVert h\rVert_s,
 \end{equation}
 since
 \[
  \lvert \varphi g\rvert_{W, \alpha, q}=\lvert W\rvert^\alpha \lvert \varphi g\rvert_{C^q(W, \mathbb C)}\le \lvert W\rvert^\alpha \lvert \varphi\rvert_{C^q(W, \mathbb C)}\lvert g\rvert_{C^1}
  =\lvert \varphi \rvert_{W, \alpha, q}\lvert g\rvert_{C^1} \le \lvert g\rvert_{C^1}.
 \]
Furthermore,
\[
  \lVert gh\rVert_u=\sup_{\epsilon \le \epsilon_0}\sup_{\substack{W_1, W_2\in \Sigma \\ d_{\Sigma} (W_1, W_2)\le \epsilon}}\sup_{\substack {\lvert \varphi_i\rvert_{C^1(W_i, \mathbb C)}
 \le 1 \\d_q(\varphi_1, \varphi_2)\le \epsilon}}\frac{1}{\epsilon^\beta}\bigg{\lvert}
 \int_{W_1}h\varphi_1g\, dm-\int_{W_2}h\varphi_2g\, dm\bigg{\rvert}.
\]
Using the notation as in~\cite[p.12]{DL} we have that
\[
 \begin{split}
  \frac{1}{\epsilon^\beta}\bigg{\lvert}
 \int_{W_1}h\varphi_1g\, dm-\int_{W_2}h\varphi_2g\, dm\bigg{\rvert} &\le   \frac{1}{\epsilon^\beta}\bigg{\lvert}
 \int_{W_1}h\varphi_1g\, dm-\int_{W_2}h((\varphi_1 g)\circ \Phi)\, dm\bigg{\rvert}\\
 &\phantom{\le}+\frac{1}{\epsilon^\beta}\bigg{\lvert}\int_{W_2}h((\varphi_1 g)\circ \Phi)\, dm -\int_{W_2}h\varphi_2g\, dm\bigg{\rvert}\\
 &=:(I)+(II),
 \end{split}
\]
where $\Phi:=G_{F_1}\circ G_{F_2}^{-1},$ and $F_1$ and $F_2$ are respectively the parametrization of $W_1$ and $W_2$ in the local charts.

Let us first estimate term $(I)$.
Note that \begin{equation}\label{uif1}
   \lvert \varphi_1 g\rvert_{C^1(W_1, \mathbb C)}\le \lvert \varphi_1\rvert_{C^1(W_1, \mathbb C)}\cdot \lvert g\rvert_{C^1(W_1, \mathbb C)}
        \le \lvert g\rvert_{C^1}.
        \end{equation}
 We now bound the term $\lvert (\varphi_1 g)\circ \Phi\rvert_{C^1(W_2, \mathbb C)}$. In the estimates that follow $C>0$ will denote
an arbitrary positive number independent of $g$ and $h$.
Observe that
\[
 \lvert (\varphi_1 g)\circ \Phi\rvert_{C^0(W_2, \mathbb C)}\le \lvert \varphi_1 g\rvert_{C^0(W_1, \mathbb C)}\le \lvert \varphi_1\rvert_{C^0(W_1, \mathbb C)}\cdot \lvert g\rvert_{C^0}\le
 \lvert \varphi_1\rvert_{C^1(W_1, \mathbb C)}\cdot \lvert g\rvert_{C^1}
 \le \lvert g\rvert_{C^1}.
\]
Furthermore, \footnote{$Lip_{q,W}(f):=\sup_{\substack{x\neq y \\ x, y\in W}} \frac{\lvert f(x)-f(y)\rvert}{\lvert x-y\rvert^q}$.}
\[
 Lip_{1, W_2}((\varphi_1 g)\circ \Phi)\le Lip_{1, W_1}(\varphi_1 g) \cdot Lip_{1, W_2}(\Phi)\le C Lip_{1, W_1}(\varphi_1 g),
\]
since $\sup_F \max \{\lvert G_F\rvert_{C^1}, \lvert G_F^{-1}\rvert_{C^1} \} <\infty$. Moreover, since
\[
 Lip_{1, W_1}(\varphi_1 g)\le \lvert \varphi_1\rvert_{C^0(W_1, \mathbb C)}Lip_{1, W_1} (g)+\lvert g\rvert_{C^0(W_1, \mathbb C)}Lip_{1, W_1} (\varphi_1),
\]
it follows that
\[
  Lip_{1, W_1}(\varphi_1 g)\le C\lvert g\rvert_{C^1}.
\]
Consequently,
\begin{equation}\label{uif2}
 \lvert (\varphi_1 g)\circ \Phi\rvert_{C^1(W_2, \mathbb C)}\le C\lvert g\rvert_{C^1}.
\end{equation}
Finally, we observe that
\begin{equation}\label{uif3}
 d_q(\varphi_1 g, (\varphi_1 g)\circ \Phi)=\lvert (\varphi_1 g)\circ G_{F_1}-(\varphi_1 g)\circ \Phi \circ G_{F_2}\rvert_{C^q(I_{r_1}, \mathbb C)}=0,
\end{equation}
since $\Phi \circ G_{F_2}=G_{F_1}$. Hence, \eqref{uif1}, \eqref{uif2} and~\eqref{uif3} imply that
\begin{equation}\label{uif4}
 (I)\le C \lvert g\rvert_{C^1}\lVert h\rVert_u.
\end{equation}

In order to estimate $(II)$, we need to bound
\[
 \lvert (\varphi_1 g)\circ \Phi-\varphi_2g\rvert_{W_2, \alpha, q}=\lvert W_2\rvert^\alpha \cdot  \lvert (\varphi_1 g)\circ \Phi-\varphi_2g\rvert_{C^q(W_2, \mathbb C)}
 \le C\lvert (\varphi_1 g)\circ \Phi-\varphi_2g\rvert_{C^q(W_2, \mathbb C)}.
\]
Note that
\[
 \begin{split}
  \lvert (\varphi_1 g)\circ \Phi-\varphi_2g\rvert_{C^q(W_2, \mathbb C)} &=\lvert (\varphi_1 g)\circ \Phi \circ G_{F_2}\circ G_{F_2}^{-1}-(\varphi_2g)\circ G_{F_2}\circ G_{F_2}^{-1}\rvert_{C^q(W_2, \mathbb C)}
 \\
 &\le C\lvert (\varphi_1 g)\circ \Phi \circ G_{F_2}-(\varphi_2g)\circ G_{F_2}\rvert_{C^q(I_{r_1}, \mathbb C)}\\
 &=  \lvert (g\varphi_1)\circ G_{F_1}-(g\varphi_2) \circ G_{F_2}\rvert_{C^q(I_{r_1}, \mathbb C)} \\
   &\le \lvert (g\circ G_{F_1})(\varphi_1 \circ G_{F_1})-(g\circ G_{F_1})(\varphi_2 \circ G_{F_2})\rvert_{C^q(I_{r_1}, \mathbb C)}\\
   &\phantom{\le}+ \lvert (g\circ G_{F_1})(\varphi_2 \circ G_{F_2})-(g\circ G_{F_2})(\varphi_2 \circ G_{F_2})\rvert_{C^q(I_{r_1}, \mathbb C)}\\
  &\le \lvert g\circ G_{F_1}\rvert_{C^q(I_{r_1}, \mathbb C)}\cdot d_q(\varphi_1, \varphi_2) \\
  &\phantom{\le}+\lvert g\circ G_{F_1}-g\circ G_{F_2}\rvert_{C^q(I_{r_1}, \mathbb C)}\cdot \lvert \varphi_2 \circ G_{F_2}\rvert_{C^q(I_{r_1}, \mathbb C)}\\
  &\le \epsilon \lvert g\circ G_{F_1}\rvert_{C^1(I_{r_1}, \mathbb C)} \\
  &\phantom{\le}+\lvert g\circ G_{F_1}-g\circ G_{F_2}\rvert_{C^q(I_{r_1}, \mathbb C)}\cdot \lvert \varphi_2 \circ G_{F_2}\rvert_{C^1(I_{r_1}, \mathbb C)}.\\
  \end{split}
\]
 Since $\sup_F\lvert G_F\rvert_{C^1(I_r, \mathbb C)}
<\infty$, we have  that
\[ \lvert g\circ G_{F_1}\rvert_{C^1(I_{r_1}, \mathbb C)} \le C\lvert g\rvert_{C^1} \quad \text{and} \quad
\lvert \varphi_2 \circ G_{F_2}\rvert_{C^1(I_{r_1}, \mathbb C)}\le C.\] Finally, it remains to estimate
\[
 \lvert g\circ G_{F_1}-g\circ G_{F_2}\rvert_{C^q(I_{r_1}, \mathbb C)}.
\]
Since $d_\Sigma(W_1, W_2)\le \epsilon$, it follows from the mean-value theorem that and setup from~\cite[p.12]{DL} that
\begin{equation}\label{el}
 \lvert g\circ G_{F_1}(t)-g\circ G_{F_2}(t)\rvert \le C\epsilon \lvert g\rvert_{C^1}  \quad \text{for $t\in I_{r_1}$.}
\end{equation}
Indeed, we have that
\[
 \begin{split}
  \lvert g\circ G_{F_1}(t)-g\circ G_{F_2}(t)\rvert &\le \lvert g\rvert_{C^1} \lvert G_{F_1}(t)-G_{F_2}(t)\rvert \\
  &=\lvert g\rvert_{C^1}\lvert \chi_j(x_1+(t, F_1(t)))-\chi_j(x_2+(t,F_2(t))) \rvert \\
  &\le C \lvert g\rvert_{C^1}(\lvert x_1-x_2\rvert +\lvert F_1(t)-F_2(t)\rvert) \\
  &\le C\lvert g\rvert_{C^1} d_\Sigma(W_1, W_2),
 \end{split}
\]
for $t\in I_{r_1}$ which implies~\eqref{el}.
Take now $t, s\in I_{r_1}$. Then, \eqref{el} implies that
\[
 \frac{\lvert g\circ G_{F_1}(t)-g\circ G_{F_2}(t)-g\circ G_{F_1}(s)+g\circ G_{F_2}(s)
\rvert}{\lvert t-s\rvert^q}\le \frac{2C\epsilon \lvert g\rvert_{C^1}}{\lvert t-s\rvert^q}.
 \]
On the other hand, by applying the mean-value theorem we have
\[
 \frac{\lvert g\circ G_{F_1}(t)-g\circ G_{F_2}(t)-g\circ G_{F_1}(s)+g\circ G_{F_2}(s)
\rvert}{\lvert t-s\rvert^q}\le \frac{2C\lvert g\rvert_{C^1}\lvert t-s\rvert}{\lvert t-s\rvert^q}.
 \]
One can now proceed as in~\cite[p.20]{DL} to show that
\[
 \sup_{t\neq s} \frac{\lvert g\circ G_{F_1}(t)-g\circ G_{F_2}(t)-g\circ G_{F_1}(s)+g\circ G_{F_2}(s)
\rvert}{\lvert t-s\rvert^q} \le C\lvert g\rvert_{C^1}\epsilon^{1-q}.
\]
This together with~\eqref{el} implies that
\[
  \lvert g\circ G_{F_1}-g\circ G_{F_2}\rvert_{C^q(I_{r_1}, \mathbb C)}\le C\lvert g\rvert_{C^1}\epsilon^{1-q}.
\]
We conclude that
\[
 \lvert (\varphi_1 g)\circ \Phi-\varphi_2g\rvert_{C^q(W_2, \mathbb C)}\le C\epsilon^{1-q} \lvert g\rvert_{C^1}\le C\epsilon^\beta \lvert g\rvert_{C^1}.
\]
Thus,
\begin{equation}\label{est2}
(II)\le C\lvert g\rvert_{C^1} \lVert h\rVert_s.
\end{equation}
The conclusion of the lemma follows directly from~\eqref{qx}, \eqref{est1}, \eqref{uif4} and~\eqref{est2}.
\end{proof}

The following proposition is analogous to Proposition \ref{twistcty}.
\begin{proposition}\label{AS}
 There exists a continuous function $K\colon \mathbb C \to (0, \infty)$ such that
 \begin{equation}\label{ubt}
  \lVert \mcl_\om^{\theta}h\rVert \le K(\theta)\lVert h\rVert, \quad \text{for $h\in \mathcal B$, $\theta \in \mathbb C$ and \paeom.}
 \end{equation}

\end{proposition}

\begin{proof}
 Note that it follows from~\eqref{ub} and Lemma~\ref{EST} that
 \[
  \lVert \mcl_\om^{\theta}h\rVert=\lVert \mcl_\om(e^{\theta g(\om, \cdot)}h)\rVert \le K \lVert e^{\theta g(\om, \cdot)}h\rVert\le CK\lvert e^{\theta g(\om, \cdot)}\rvert_{C^1}
  \lVert h\rVert,
 \]
for $h\in \mathcal B$, $\theta \in \mathbb C$ and \paeom. Furthermore, observe that~\eqref{obs} implies that
\[
 \lvert e^{\theta g(\om, \cdot)}\rvert_{C^0}\le e^{M\lvert \theta\rvert} \quad \text{for \paeom.}
\]
Similarly, it follows from the mean-value theorem (applied for a map $z\mapsto e^{\theta z}$) and~\eqref{obs} that
\[
 \sup_{x\neq y} \frac{\lvert e^{\theta g(\om, x)}-e^{\theta g(\om, y)}\rvert}{\lvert x-y\rvert} \le \lvert \theta\rvert e^{2M\lvert \theta\rvert}
 \sup_{x\neq y} \frac{\lvert g(\om, x)-g(\om, y)\rvert}{\lvert x-y\rvert}\le M\lvert \theta\rvert e^{2M\lvert \theta\rvert}.
\]
The desired conclusion follows directly from the above estimates.

\end{proof}

Analogously to Proposition \ref{twistQC} we have:
\begin{proposition}
 For $\theta$ close to $0$, the cocycle $(\mcl_\om^\theta)_{\om \in \Om}$ is quasicompact.
\end{proposition}

\begin{proof}
 We follow closely~\cite[Lemma 3.13]{DFGTV}. Observe~\eqref{wsly} and choose $N\in \mathbb N$ such that $\gamma:=Ba^N<1$. Hence,
 \[
 \begin{split}
  \lVert \mcl_\om^{\theta, (N)} h\rVert &\le \lVert  \mcl_\om^{ (N)} h\rVert+\lVert \mcl_\om^{\theta, (N)}-\mcl_\om^{(N)}\rVert \cdot \lVert h\rVert \\
  &\le
  \gamma \lVert h\rVert+B\lvert h\rvert_w+\lVert \mcl_\om^{\theta, (N)}-\mcl_\om^{(N)}\rVert \cdot \lVert h\rVert.
  \end{split}
 \]
 On the other hand, we have that
\[
 \mcl_\om^{\theta, (N)}-\mcl_\om^{(N)}=\sum_{j=0}^{N-1} \mcl_{\sigma^{N-j} \om}^{\theta, (j)}(\mcl_{\sigma^{N-1-j} \om}^\theta -\mcl_{\sigma^{N-1-j} \om})\mcl_\om^{(N-1-j)}.
\]
It follows from~\eqref{ub} and~\eqref{ubt} that
\[
 \lVert \mcl_\om^{(N-1-j)}\rVert \le K^{N-1-j} \quad \text{and} \quad \lVert \mcl_{\sigma^{N-j} \om}^{\theta, (j)}\rVert \le K(\theta)^j.
\]
Furthermore, using~\eqref{ub} and Lemma~\ref{EST}, we have that for any $h\in \mathcal B$ and \paeom,
\[
 \lVert (\mcl_\om^\theta -\mcl _\om)(h) \rVert=\lVert \mcl_\om (e^{\theta g(\om, \cdot)}h-h)\rVert \le K\lVert (e^{\theta g(\om, \cdot)}-1)h\rVert
 \le CK \lvert e^{\theta g(\om, \cdot)}-1\rvert_{C^1} \lVert h\rVert.
 \]
 On the other hand, using~\eqref{obs} and applying the mean value theorem for the map $z\mapsto e^{\theta z}$, it is easy to verify that there exists $C'>0$ such that
for $\theta \in B_{\mathbb C}(0, 1)$,
\begin{equation}\label{ao2}
  \lvert e^{\theta g(\om, \cdot)}-1\rvert_{C^1} \le C'\lvert  \theta \rvert  \quad \text{for \paeom.}
\end{equation}
Hence, there exists $\tilde C>0$ such that
\[
 \lVert \mcl_\om^\theta -\mcl _\om\rVert \le \tilde C\lvert \theta \rvert, \quad \text{for \paeom.}
\]
We conclude that
\[
 \lVert \mcl_\om^{\theta, (N)}-\mcl_\om^{(N)}\rVert \le \tilde C\lvert  \theta \rvert \sum_{j=0}^{N-1}K^{N-1-j}K(\theta)^j,
\]
and therefore there exists $\tilde \gamma \in (0, 1)$ such that for any $\theta$ sufficiently close to $0$ and $h\in \mathcal B$,
\begin{equation}\label{0320}
 \lVert \mcl_\om^{\theta, (N)} h\rVert \le \tilde \gamma \lVert h\rVert+B\lvert h\rvert_w.
\end{equation}
Similarly, one can show that there exists $\tilde B>0$ such that for any $\theta$ sufficiently close to $0$ and $h\in \mathcal B$,
\begin{equation}\label{0321}
 \lvert \mcl_\om^\theta h\rvert_w\le \tilde B\lvert h\rvert_w.
\end{equation}
The conclusion of the proposition follows from~\eqref{0320} and~\eqref{0321} by arguing as in~\cite[Theorem 3.12]{DFGTV}.
\end{proof}

\subsection{Regularity of the top Oseledets space, convexity of $\Lambda$}
The regularity of the top Oseledets space of the twisted cocycles follows identically as in Section \ref{sec:regularity}, with Lemma~\ref{EST} used in place of Lemma~3.2~\cite{GL} in the proof of Lemma~\ref{lem:analyt}. Moreover the family of probability measures $h^0_{\omega}$ will allow us to define the fibred measure $\mu_{\omega}$ as we did in Section \ref{sec:convexity}.

\subsection{Large deviation principle and central limit theorem}

The results of Sections \ref{sec:Lambda} and \ref{sec:ldp_clt} follow verbatim with the obvious modifications.
We thus obtain our main results for piecewise hyperbolic dynamics.

\begin{mainthm}[Quenched large deviations theorem]\label{thm:ldtpwhyp}
In the setting of Section \ref{sect:pwh}, there exists $\epsilon_0>0$ and a non-random function $c\colon (-\epsilon_0, \epsilon_0) \to \mathbb R$ which is nonnegative, continuous, strictly convex, vanishing only at $0$ and such that
\[
\lim_{n\to \infty} \frac 1 n \log \mu_\om(S_n g(\om, \cdot ) >n\epsilon)=-c(\epsilon), \quad \text{for  $0<\epsilon <\epsilon_0$ and \paeom}.
\]
\end{mainthm}
\begin{mainthm}[Quenched central limit theorem] \label{thm:cltpwhyp}
In the setting of Section \ref{sect:pwh}, assume that the non-random \textit{variance} $\Sig^2$, defined in \eqref{variance} satisfies $\Sig^2>0$.
Then, for every bounded and continuous function $\phi \colon \R \to \R$ and \paeom, we have
\[
\lim_{n\to\infty}\int \phi \bigg{(}\frac{S_n g(\om, x)}{ \sqrt n}\bigg{)}\, d\mu_\om (x)=\int \phi \, d\mathcal N(0, \Sig^2).
\]
(The discussion in \S\ref{sec:convexity} deals with the degenerate case $\Sig^2=0$).
\end{mainthm}
\subsection{Local central limit theorem}
\begin{mainthm}[Quenched Local central limit theorem] \label{thm:lclt_piec}
In the setting of Section \ref{sect:pwh}, suppose that condition (L) holds, where the functional norm in (L) is now $\mathcal{B}.$

Then, for \paeom\  and every bounded interval $J\subset \R$, we have
 \[
  \lim_{n\to \infty}\sup_{s\in \R} \bigg{\lvert} \Sig \sqrt{n} \mu_\om (s+S_n g(\om, \cdot)\in J)-\frac{1}{\sqrt{2\pi}}e^{-\frac{s^2}{2n\Sig^2}}\lvert J\rvert \bigg{\rvert}=0.
 \]
\end{mainthm}
The LCLT can also be obtained under the assumptions HK A1, A2, A3 and the hypothesis of Lemma \ref{EQUI} with the obvious change of the functional space which is now $\mathcal{B}.$ The Lasota--Yorke inequality for the twisted operator follows now  by adapting the analogous proof in \cite{DL} for the usual operator. So we have the analogous statement as in Theorem \ref{HKLCLT}.
\subsection{Billiards}
The results of this section  also apply  to the billiard map associated with both a finite and infinite horizon Lorentz gas having smooth scatterers with strictly positive curvature: we refer in the following to the papers by Demers and Zhang \cite{DM1, DM2}.\\

  We first recall the setting. Let us consider on the bidimensional torus $\mathbb{T}^2$ a finite number of pairwise disjoint and simply connected convex regions $\{\Gamma\}_{i=1}^d,$ which moreover have $C^3$ boundary curves $\partial \Gamma_i$ with strictly positive curvature. We denote by $\text{int}A$ the interior of the set $A;$ then the billiard table $Q$ is defined as $Q=\mathbb{T}^2\backslash \cup_i\text{int}\Gamma_i.$
On the phase space  $\mathcal{M}=Q\times \mathbb{S}^1/\sim,$ with the conventional identifications at the boundaries, we define the billiard flow, which is  induced by a particle traveling at unit speed and undergoing elastic collisions at the boundaries.
We will be concerned instead with the billiard map $T:M\rightarrow M$ as the Poincar\'e map corresponding to collisions with scatterers and defined on $M=\cup_i\partial \Gamma_i\times[-\pi/2, \pi/2].$ We put on $M$ the coordinates $(r, \theta),$ where $r\in \cup_i\partial \Gamma_i$ is parametrized by arc length and $\theta$ is the angle formed by the unit tangent vector at $r$ with the normal pointing into the domain $Q.$ The map $T$ preserves a probability measure $\mu$ defined by $d\mu= c \cos \theta dr d\theta$, where $c$ is the normalizing constant. If we denote by $\tau(x)$ the first non-tangential collision time of the orbit flow starting at $x$, then $T$ is defined when $\tau(x)<\infty$ and in this case it is uniformly hyperbolic. In particular $T$ has a finite horizon if $\tau$ is bounded from above, otherwise $T$ has infinite horizon. The map $T$ shares the same properties of the piecewise hyperbolic maps studied in this section, with a relevant difference: its derivative $DT$ becomes infinite near singularities. This fact will induce a slight change in the definition of the norms. The latter are defined exactly as in section 9 but the norm on the test function $\varphi$  (\ref{TF}) is now modified as
$$
|\varphi|_{W, \alpha, q}=|W|^{\alpha} \cdot \cos W \cdot |\phi|_{C^q(W, \mathbb{C})},
$$
where $\cos W=\frac{1}{m_W(W)}\int_W \cos\theta dm_W,$ being $m_W$ the unnormalized Lebesgue measure on the stable curve $W.$ With this precaution, the Banach space $\mathcal{B}$ is defined as the completion of $C^1(M, \mathbb{C})$ with respect to the norm given by the sum of the strong norm defined in (\ref{ssnorm}) and the strong unstable norm given by (\ref{sunorm}).  In \cite{DM1},  Demers and Zhang established the Lasota-Yorke inequality and the associated spectral picture in the deterministic setting of a single billiard map $T$. The main technical  difference with \cite{DL} was the control of distortion which required additional cuts at the boundaries of homogeneity strips with the consequence of generating a countably infinite number of curves in both the finite and infinite horizon cases. In \cite{DM2} the same authors introduced a distance between  maps, see section 3.4 in \cite{DM2}. Then they  consider a family $\mathcal{F}$ of billiard maps such that by taking composition of maps close in $\mathcal{F},$ that composition has the same hyperbolic and distortion properties of the iterates of a single map, see the discussion in section 5.3 in \cite{DM2}. Moreover they defined a random walk on $M$ by choosing the sequence of maps in $\mathcal{F}$ in an i.i.d. way with a prescribed density (see their $g(\omega,\cdot)$ in section 2.3), where $\omega$ belongs to the probability space $(\Omega, \nu).$ By averaging the Perron-Fr\"obenius operators associated to the maps in the sequence over $\nu$ they finally defined an averaged transfer operator and applied to it standard perturbation theory, thus getting the {\em annealed} limit theorems stated in their Theorem 2.6. Our approach is devoted to quenched results; for that and by eventually  reducing  the family $\mathcal{F}$ by the Conze-Raugi criterion which we already employed in the previous sections, we get the quenched Lasota-Yorke inequality and the exponential decay expressed by the bound (\ref{DEC}), which are the bases of our theory. Theorems A,B and C then follows for the billiards maps associated with perturbations of the periodic Lorentz gas described above.

\section{Acknowledgements}
DD was supported by Croatian Science Foundation under the project IP-2014-09-2285 and by the University of Rijeka under the project
number 17.15.2.2.01.
GF and CGT thank AMU, CPT and CIRM (Marseille) for hospitality during this research.
GF is partially supported by an ARC Discovery project. CGT is supported by an ARC DECRA.
SV was supported  by the MATH AM-Sud Project
Physeco, and by the project APEX Syst\`emes dynamiques: Probabilit\`es et Approximation
Diophantienne PAD funded by the R\'egion PACA (France). SV warmly thanks the
LabEx Archim\'ede (AMU University, Marseille), and INdAM (Italy).

\end{document}